\documentclass[11pt]{amsart}

\usepackage{graphicx} 
\usepackage[english]{babel}     
\usepackage[utf8]{inputenc}
\usepackage{amsmath}
\usepackage{amsfonts}
\usepackage{amssymb}
\usepackage{hyperref}
\usepackage{xcolor}
\usepackage{comment}
\usepackage{enumerate}
\usepackage{tikz-cd}
\usepackage{doi}
\usepackage[normalem]{ulem}
\usepackage{graphicx}

\usepackage{todonotes}
\usepackage{comment}
\specialcomment{auxproof}
{\mbox{}\newline\textbf{BEGIN: AUX-PROOF}\dotfill\newline}
{\mbox{}\newline\textbf{END: AUX-PROOF}\dotfill\newline}

\usepackage{amsthm}
\newtheorem{theorem}{Theorem}[section]
\newtheorem{theoremx}{Theorem}

\newtheorem{lemma}[theorem]{Lemma}
\newtheorem{proposition}[theorem]{Proposition}
\newtheorem{corollary}[theorem]{Corollary}

\theoremstyle{definition}
\newtheorem{definition}[theorem]{Definition}
\newtheorem{example}[theorem]{Example}
\newtheorem{remark}[theorem]{Remark}

\let\emptyset\varnothing

\newcommand{\sign}{{\rm sign}}

\newcommand{\Tr}{\textrm{Tr}}
\newcommand{\TripleVert}{|\!|\!|}
\newcommand{\frakm}{\mathfrak{m}}
\newcommand{\HMS}{\textrm{HMS}}
\newcommand{\Sphere}{\mathbb{S}}
\newcommand{\NNZero}{\mathbb{N}_{\geq 0}}
\newcommand{\NNOne}{\mathbb{N}_{\geq 1}}

\newcommand{\supp}{{\rm supp}}
\newcommand{\cS}{\mathcal{S}}
\newcommand\numberthis{\addtocounter{equation}{1}\tag{\theequation}}
\title[Best constants of Schur multipliers of higher order divided differences]{On the best constants of Schur multipliers of higher order divided difference functions}
\author{Martijn Caspers}
\author{Jesse Reimann}

\address{TU Delft, EWI/DIAM,
	P.O.Box 5031,
	2600 GA Delft,
	The Netherlands}

\email{M.P.T.Caspers@tudelft.nl}

\email{J.Reimann@tudelft.nl}

\date{\today}

\begin{document}

\begin{abstract}
Let $f \in C^n(\mathbb{R})$ be such that $\Vert f^{(n)} \Vert_\infty < \infty$. Let $f^{[n]} \in C(\mathbb{R}^{n+1})$ be the $n$th order divided difference. A special case of our main result states that for $1 < p < \infty$ we have
\[
\Vert T_{f^{[n]}}: S_{np} \times \ldots \times S_{np} \rightarrow S_{p} \Vert \lesssim p^\ast p^n \Vert f^{(n)} \Vert_\infty,
\]
where $p^\ast = p/(p-1)$ is the H\"older conjugate of $p$ and $T_{f^{[n]}}$ is the multilinear Schur multiplier with symbol $f^{[n]}$.  
In case of the generalized absolute value map $f(\lambda) = \lambda^{n-1} \vert \lambda \vert, \lambda \in \mathbb{R}$, we show that 
\[
p^\ast p^{n}  \lesssim \Vert T_{f^{[n]}}: S_{np} \times \ldots \times S_{np} \rightarrow S_{p} \Vert.
\]
This provides an alternative proof to one of the key theorems in  the solution of Koplienko’s problem on higher order spectral shift [Invent.\ Math.\ 193, No.\ 3, 501--538 (2013)], which is moreover sharp as $p \searrow 1$   and as $p \rightarrow \infty$, for any $n$.
\end{abstract}

\maketitle

\section{Introduction}

The theory of \emph{spectral shift functions}, introduced by Lifschitz~\cite{Lifschitz} and subsequently developed by Krein~\cite{Krein1,Krein2}, concerns functions~$\xi_{H_0,V}$ such that \begin{equation*}
    \mathrm{Tr}(f(H_0+V)-f(H_0))=\int_{\mathbb{R}}f'(t)\xi_{H_0,V}(t)dt
\end{equation*} for suitable functions $f$, self-adjoint operators $H_0$, and perturbations $V$. This theory has found applications in a wide range of areas, such as spectral theory, perturbation theory, mathematical physics, noncommutative geometry, and many more.  See~\cite{Gesztesy, SkripkaSurvey} for a historical overview and~\cite{BirmanH}, \cite[Section 5.5]{SkripkaTomskova} for a survey and applications of spectral shift. In~\cite{kopl_trace}, Koplienko conjectured the existence of \emph{higher order spectral shift functions}, i.e.\ integrable functions~$\eta_{n, H_0,V}$ such that \begin{equation}\label{eqn: koplienko}
    \mathrm{Tr}\left(f(H_0+V)-\sum_{k=0}^{n-1}\frac{1}{k!}\frac{d^k}{dt^k}f(H_0+tV)\Bigg|_{t=0}\right)=\int_{\mathbb{R}}f^{(n)}(t)\eta_{n,H_0,V}(t)dt
\end{equation}
is well-defined for sufficiently nice $f$, $H_0$ self-adjoint, and $V$ in the $n$th Schatten class $S_n$, $n\ge 2$. 
\vspace{0.3cm}

After partial results for $n=2$~\cite{kopl_trace}, and for general $n\in\mathbb{N}$ if $V$ is a Hilbert-Schmidt operator (i.e.\ if $V\in S_2$)~\cite{dykema_skripka_jfa, skripka_indiana, skripka_illinois}, the full conjecture was resolved affirmatively in 2013 by Potapov, Skripka, and Sukochev~\cite{PSS-Inventiones}. Key to their proof is expressing the left-hand side of~\eqref{eqn: koplienko} as a \emph{multiple operator integral}, after which the existence of the spectral shift function follows from a Riesz representation theorem argument. Problems concerning boundedness of multiple operator integrals can be translated to  boundedness problems of closely related {\it multilinear Schur multipliers}. No general theorems concerning boundedness of multilinear Schur multipliers are available and therefore such problems can become notoriously difficult. The key result in the main body of~\cite{PSS-Inventiones} is boundedness of the relevant multilinear Schur multipliers of higher order \emph{divided differences}
\begin{align*}
    f^{[1]}(\lambda_0,\lambda_1)&=\begin{cases}
        \frac{f(\lambda_0)-f(\lambda_1)}{\lambda_0-\lambda_1},&\lambda_0\neq \lambda_1,\\
        f'(\lambda_0),&\lambda_0=\lambda_1,
    \end{cases},\\
    f^{[n]}(\lambda_0,\dotsc,\lambda_n)&=\begin{cases}
        \frac{f^{[n-1]}(\lambda_0,\dotsc,\lambda_{n-1})-f^{[n-1]}(\lambda_1,\dotsc,\lambda_n)}{\lambda_0-\lambda_n},&\lambda_0\neq\lambda_n, \\
        \frac{1}{n!} f^{(n)}(\lambda_0),&\lambda_0=\dotsc=\lambda_n.
    \end{cases}
\end{align*}
{Boundedness of these multilinear Schur multipliers is shown by a deep iterative decomposition into linear Schur multipliers with symbols of a particular type, namely \emph{Toeplitz-form} (i.e. Schur multipliers with symbols of the form $m(\lambda,\mu)=\tilde{m}(\lambda-\mu)$)  and {first order} \emph{polynomial integral momenta}, of which first order divided differences are a special case. The boundedness of these linear operators implies the boundedness of the multilinear Schur multiplier of divided differences and hence their main theorem.
} 

\vspace{0.3cm}

In the first order case, significant interest was shown in the optimal growth rate of the boundedness constants of the linear Schur multiplier of first order divided differences of Lipschitz functions~$f$, acting on Schatten classes $S_p$. Indeed, when boundedness of this Schur multiplier was shown in~\cite{PoSu11} for $1<p<\infty$, it was already known that these operators are unbounded on~$S_1$ and $S_{\infty}$~\cite{Farforovskaya1,Farforovskaya2,Farforovskaya3}, disproving a conjecture by Krein~\cite{KreinConjecture}. This inspired the search for optimal growth rates of the boundedness constant in $p$ as $p\searrow 1$ and $p\to\infty$, which were {eventually}  found in~\cite{CMPS,CJSZ,CPSZ}, including endpoint estimates in~$\mathrm{BMO}$ and $S_{1,\infty}$. In particular, the endpoint estimate in $S_{1, \infty}$ confirmed  the Nazarov-Peller conjecture~\cite{NazarovPeller}.  
\vspace{0.3cm}

In the higher order case it was shown in~\cite{CLPST,PSST17} that multilinear Schur multipliers of divided differences generally do not map boundedly to $S_1$ under the minimal regularity assumptions we use in the current paper, where the counter example is a modification of the generalized absolute value function that we define in Section~\ref{Sect=LowerBounds}. For this specific function the endpoint estimate in $S_{1, \infty}$ was obtained in~\cite{CSZ-Israel}. We also mention that by imposing stricter regularity assumptions on~$f$, such as Besov regularity, bounds on multilinear Schur multipliers were obtained in~\cite{CaspersHuisman} when the recipient space is $S_{1}$ or even a quasi-Banach $S_p$-space $0< p < 1$. The latter result builds on earlier contributions by Peller \cite{Peller} and McDonald-Sukochev \cite{McDonaldSukochev} in the linear case.

\vspace{0.3cm}

Recent progress in the study of Schur multipliers (see Parcet's survey at the {ICM 2026}~\cite{ParcetICM}) has motivated a re-examination of the boundedness proof of the multilinear Schur multiplier of higher order divided differences in~\cite{PSS-Inventiones}, with the goal of obtaining optimal growth rates as~$p\searrow 1$ or $p\to \infty$ in the  recipient space  $S_p$. Recent results such as the H\"ormander-Mikhlin-Schur multiplier theorem~\cite{CGPT} and Marcinkiewicz multiplier theorem~\cite{marcinkiewicz_schur} give new sufficient conditions for the boundedness of a large class of linear Schur multipliers. Moreover, transference methods for multilinear Toeplitz-form Schur multipliers~\cite{CJKM, CKV} allow for the application of results from multilinear harmonic analysis and reduce the number of needed decomposition steps. In~\cite{CaspersReimann}, the bilinear case was studied and the optimal growth rate as~$p\searrow 1$ was found. While the authors were unable to find the optimal growth rate as $p\to\infty$, the use of recent advances in multilinear harmonic analysis yielded a narrow range of possible growth rates, significantly improving the growth rate one obtains by following the proof in~\cite{PSS-Inventiones}.

\vspace{0.3cm}

\noindent {\bf Main result.} In order to state our main result, we first give informal definitions of the central objects appearing in these theorems. For details, see Section~\ref{sect: prelims}.

	 Given a sufficiently smooth function $f$, we can define the $n$-linear \emph{Schur multiplier of its $n$th order divided difference} $T_{f^{[n]}}$, acting on operators in \emph{Schatten classes} $S_p$. These spaces $S_p$ consist of compact operators such that their singular values form an $\ell^p$-sequence, {$1 \leq p<\infty$}. On integral operators with explicit kernels $A_1,\dotsc,A_n$, the Schur multiplier acts by multiplying the kernels of the operator with its \emph{symbol} $f^{[n]}$, yielding an integral operator with kernel \begin{equation*}
T_{f^{[n]}}(A_1, \ldots, A_n)(s_0, s_n) = \int_{\mathbb{R}^{n-1}} f^{[n]}(s_0, \ldots, s_n) A_1(s_0, s_1) \ldots A_n(s_{n-1}, s_n) ds_1 \ldots ds_{n-1}. 
	 \end{equation*} 
     Finally, for $1<p<\infty$ let $p^*:=p/(p-1)$. We are now ready to state our main theorems.
	
	\begin{theoremx}[Upper bound]\label{Thm=TheoremA} 
For every $f \in C^n(\mathbb{R})$ such that $\Vert f^{(n)} \Vert < \infty$ and for every ${1 < p, p_1, \ldots, p_n < \infty}$ with $p^{-1} = \sum_{i=1}^n p_i^{-1}$ we have
\begin{equation*}
 \|T_{f^{[n]}}:S_{p_1}\times\dotsc\times S_{p_n}\to S_p\|   \leq D_n(p,p_1,\dotsc,p_n) \Vert f^{(n)}\Vert_\infty,
\end{equation*}
where $D_n(p,p_1,\dotsc,p_n)<\infty$ is independent of $f$. If $p_1=\dotsc=p_n=np$, then $$D_n(p,np,\dotsc,np)\le C_n p^\ast p^n$$ for some $C_n<\infty$ independent of $p$. 
\end{theoremx}
In the linear case this bound has first been obtained in~\cite{CMPS}, after which alternative proofs were found in~\cite{CPSZ,CJSZ,CGPT}, see also \cite{RieszSchur}. For $n\ge2$, Theorem~\ref{Thm=TheoremA} is new.
{The general explicit description of $D_n(p,p_1,\dotsc,p_n)$ is given in the statement of Theorem~\ref{Thm=MainTheorem}, where~$D_n(p,p_1,\dotsc,p_n)$ is a product of $n$ exponents and each individual exponent is a partial H\"older combination of $p_1, \ldots, p_n$.} 
Combined with Theorem~\ref{Thm=TheoremB} below, we see that the asymptotic behaviour of this upper bound is optimal as $p\searrow 1$ and as~$p\to\infty$ for any $n\in\mathbb{N}$.
\vspace{0.3cm}

Similar to~\cite{PSS-Inventiones,CaspersReimann}, at the heart of our proof is the repeated application of the decomposition formula
\begin{equation}\label{eqn: intro_insert_xi}
        f^{[n]}(\lambda_0,\dotsc,\lambda_n) = \frac{\lambda_i-\lambda_k}{\lambda_i-\lambda_j}f^{[n]}(\lambda_0, \dotsc,\lambda_i,\lambda_k,\lambda_k, \dotsc, \lambda_n) 
         + \frac{\lambda_k-\lambda_j}{\lambda_i-\lambda_j}f^{[n]}(\lambda_0, \dotsc,\lambda_k,\lambda_k,{\lambda_j},\dotsc, \lambda_n), 
\end{equation}
where we stress that divided differences are invariant under permutation of variables, and hence this formula holds regardless of the choice of $i\neq j$ and $k$. By repeatedly applying  this formula,  an $n$th order divided difference in $n+1$ variables can be decomposed into $n$th order divided differences with repeating variables, in particular into those only depending on two pairwise distinct variables. This allows us to treat them as the symbol of a \emph{linear} Schur multiplier, which we can estimate using the H\"ormander-Mikhlin-Schur theorem of~\cite{CGPT}.
\vspace{0.3cm}

In contrast to~\cite{CaspersReimann} however, we avoid the use of multilinear transference techniques to bound the Schur multipliers associated with the fractions in~\eqref{eqn: intro_insert_xi}, which allows us to obtain sharper bounds. Instead, we construct a partition of unity on $\mathbb{R}^{n+1}\setminus\{0\}$ such that after repeatedly applying~\eqref{eqn: intro_insert_xi} we ensure that i) {all fractions of differences in the coordinates of $\lambda$ are bounded on the support of a partition of unity function,} ii) the obtained terms act as symbols of a well-defined composition of \emph{linear} Schur multipliers that are properly nested in the sense that we can apply the reduction of \cite[Lemma 3.2 (iii)]{PSS-Inventiones}, and iii) the variables in the numerator of one of the fractions agree with the remaining two distinct variables in the divided difference (see Remark~\ref{Rmk=Matching}). As this construction is highly dependent on the choice of $i,j,k$ each time~\eqref{eqn: intro_insert_xi} is applied, we formalise all allowed choices as a collection of suitable sequences and let our partition of unity to be indexed by these choice sequences.
\vspace{0.3cm}

In order to decompose the multilinear Schur multiplier with symbol given by a product of fractions obtained from~\eqref{eqn: intro_insert_xi}, we prove a multilinear generalisation of~\cite[Lemma 6]{PoSu11}, which states that for $0<\lambda<2\mu$ one has
\begin{equation}\label{Eqn=TheoryOfEverything} 
    \frac{\lambda}{\mu}=\int_{\mathbb{R}}g(s)\lambda^{is}\mu^{-is}ds
\end{equation}
for some suitable Schwartz function $g$. Using similar methods, we show that if a smooth homogeneous function $\phi:\mathbb{R}^{n}\to\mathbb{C}$ is supported away from the coordinate axes $\{\xi \in \mathbb{R}^n \mid \xi_k=0\}$, $k=1,\dotsc,n-1$, and vanishes if $\xi_n=0$, then  we have for $\xi_1, \ldots, \xi_n \not =0$,  
\begin{multline*}
    \phi(\xi_1,\dotsc,\xi_n)
    =
     \sum_{ \epsilon \in \{ 0, 1 \}^n  } \left( \prod_{i=1}^n {\rm sign}( \xi_i )^{\epsilon_i}  \right)
    \int_{\mathbb{R}^{n-1}}
     g_{\epsilon}(s_1, \ldots,  s_{n-1})\\  \times \vert \xi_1 \vert^{ -is_1} \vert \xi_2 \vert^{ i(s_1 - s_2)}  \ldots   \vert \xi_{n-1} \vert^{i(s_{n-2}-s_{n-1})}  \vert \xi_n \vert^{ is_{n-1}}    ds_1 \ldots  ds_{n-1}  
\end{multline*}
for suitable Schwartz functions $g_\epsilon$.  What finally allows us to prove Theorem~\ref{Thm=TheoremA} with only $n$ applications of the H\"ormander-Mikhlin-Schur theorem is the fact that our construction ensures that a term of the form $(\lambda,\mu)\mapsto |\lambda-\mu|^{is}f^{[n]}(\lambda,\dotsc,\lambda,\mu,\dotsc,\mu)$ always appears in our decomposition after applying the formula above.  Moreover, note that the fact that the closed support of $\phi$ may intersect one coordinate axis is what allows us to partition all of $\mathbb{R}^{n+1}\setminus\{0\}$ into subsets on which we can apply this method. In~\cite{CaspersReimann} this was not available, hence a different proof method (yielding less sharp upper bounds) using the theory of multilinear Calder\'on-Zygmund operators was used around the coordinate axes.
\vspace{0.3cm}

{Combining all decompositions of functions above we are able to write our multilinear Schur multiplier of divided differences as a sum of compositions of $n$ different linear Schur multipliers which satisfy the criteria of the H\"ormander-Mikhlin-Schur multiplier theorem~\cite{CGPT}. This then allows us to conclude Theorem~\ref{Thm=TheoremA}.}

\vspace{0.3cm}

 To give a lower bound on the asymptotics of the operator norm of the Schur multiplier of divided differences as $p\searrow 1$ and $p\to\infty$, we focus on the special case of generalised absolute value functions $a_n(s):=s|s|^{n-1}$ and $p_1=\dotsc=p_n=np$, $1<p<\infty$. Whereas the cases $n=1,2$ were already proven in~\cite{CMPS,CaspersReimann}, the following result is new if $n\ge 3$.

\begin{theoremx}[Lower bound]\label{Thm=TheoremB} There exists a constant $C_n>0$ such that for all $1<p<\infty$,  \begin{equation*}
    C_n p^*p^{n}\le \|T_{a_n^{[n]}}:S_{np}\times \dotsc\times S_{np}\to S_p\|.
\end{equation*}
\end{theoremx}
As in~\cite{CaspersReimann}, our proof relies on constructions concerning~\emph{triangular truncations}~$T^{\pm}$, i.e.\ linear Schur multipliers with symbol $(s_0,s_1)\mapsto \chi_{>0}(\pm(s_1-s_0))$, where $\chi_{>0}$ denotes the indicator function of the positive real line. It is well-known that their operator norm on Schatten classes~$S_p$, $1<p<\infty$, is bounded from below by $p^*p$, which follows from Krein's analysis of the singular values of the Volterra operator~\cite{GohbergKrein}.

\vspace{0.3cm}

Similar to~\cite[Lemma 10]{Davies}, we exploit the discontinuity of the $n$th derivative of the generalised absolute value function $a_n$ around zero in order to approximate linear combinations and compositions of triangular truncations by Schur multipliers with the symbol $a_n^{[n]}$ restricted to certain lattices in $[-1,1]^{n+1}$. In particular, it is easy to see that $a_n^{[n]}$ is constant on the positive and negative cones of $\mathbb{R}^{n+1}$, making this function suitable for the explicit constructions needed in this proof. Compared to \cite{CaspersReimann}, the limiting case $p \rightarrow \infty$ requires an extra application of Cotlar's identity.

\vspace{0.3cm}

Combining Theorems~\ref{Thm=TheoremA} and~\ref{Thm=TheoremB} yields optimal asymptotics of order $p^*=\frac{p}{p-1}$ as $p\searrow 1$,   and of order $p^n$ as $p\to\infty$.

\vspace{0.3cm}

\noindent {\bf Structure of the paper.} After setting up notation and discussing preliminaries in Section~\ref{sect: prelims}, we turn towards the proof of Theorem~\ref{Thm=TheoremA}. In Section~\ref{Sect=ReductionFormula}, we give an explicit decomposition of higher order divided differences into Toeplitz-form fractions and divided differences in fewer variables, as is required for our proof strategy. In Section~\ref{Sect=Homogeneous}, we prove a generalisation of~\cite[Lemma~6]{PoSu11} to multivariable homogeneous functions. In Section~\ref{Sect=DivDiffDecomposition}, we construct a partition of unity on $\mathbb{R}^n$ that is adapted to both the decomposition of $f^{[n]}$ in Section~\ref{Sect=ReductionFormula} and the boundedness theorem in Section~\ref{Sect=Homogeneous}. In Section~\ref{sect: upper_bound_main_result} we prove Theorem~\ref{Thm=TheoremA}, and in Section~\ref{Sect=LowerBounds} the proof of Theorem~\ref{Thm=TheoremB} is given.

\vspace{0.3cm}

\noindent {\bf Acknowledgements.} We thank Anna Skripka for providing us with further background on Koplienko's problem. We thank Mark Veraar for useful discussions and references. An earlier version of this paper contained the weaker lower bound $p^\ast p^2$ in Theorem B for any $n \geq 2$. We are grateful to Google Gemini for suggesting us to use Cotlar's identity to improve on this bound for $n=3$ which eventually led to Theorem B.

\section{Preliminaries}\label{sect: prelims}

  For the general theory of noncommutative $L^p$-spaces we refer to~\cite{GoldsteinLabuschagne}. For multilinear operator integrals we refer to~\cite{SkripkaTomskova}, which also has a treatment of divided differences and spectral shift.

\subsection{General notation}\label{subsect: notation}
We use the following notational  conventions in this paper.
\begin{itemize}
\item $\vert I \vert$  is the cardinality of a set $I$. 
    \item $\NNZero$ are the natural numbers starting from 0. 
     \item $\mathbb{N}_{\geq 1}$ are the natural numbers starting from 1. 
     \item  $\mathbb{N}_{\geq 0}^{I}$ and $\mathbb{R}^I$ are the vectors with values in $\mathbb{N}_{\ge0}$ (resp.\ $\mathbb{R}$) of length $|I|$ for some (finite) set~$I$, understood to be indexed by the elements of $I$.
    \item A hat above a variable indicates that the  variable is omitted. 
    \item For $p \in (1, \infty)$ we write $p^\ast = \frac{p}{p-1}$ for the conjugate exponent and $p^\sharp = \max(p, p^\ast)$.
    \item   $\{ \lambda = \mu \}$ denotes the set  $\{ (\lambda, \mu) \in \mathbb{R}^2 \mid \lambda   = \mu\}$.
    \item $C_b(\mathbb{R})$ denotes the continuous bounded real valued functions on $\mathbb{R}$.
    \item  $C^{n}(\mathbb{R})$ denotes the $n$ times continuously differentiable complex valued functions.  
    \item $f^{(n)}$ is the $n$th order derivative of $f \in C^n(\mathbb{R})$.
\end{itemize}

 We write $A \lesssim B$ for an inequality that holds up to a constant that is independent of $B$. 
 In Section~\ref{Sect=LowerBounds}, we will also say that $A$ is \emph{of order} $\mathcal{O}(B)$ to denote $A\lesssim B$, and write $A\approx B$ if $A\lesssim B$ and $B\lesssim A$ both hold.

\subsection{Homogeneous sets}  
A set $A \subseteq \mathbb{R}^{n}$ is called   homogeneous  if for every $\mu > 0$ and $\lambda \in A$ we have $\mu \lambda \in A$. A function $f: \mathbb{R}^{n} \rightarrow \mathbb{C}$ is homogeneous if for all $\lambda \in \mathbb{R}^n$, $\mu>0$ we have~${f(\mu \lambda) = f(\lambda)}$. A function $f$ is   even  if $f(-\lambda) = f(\lambda)$ and odd  if $f(-\lambda) = -f(\lambda)$ for all~${\lambda \in \mathbb{R}^{n} \backslash \{ 0 \}}$.

\subsection{Divided differences}

\begin{definition}\label{Dfn=DivDiff} 
Let $f \in C^n(\mathbb{R})$ and $0 \leq k \leq n$. We define the $k$th order divided difference of $f$, denoted by~$f^{[k]}: \mathbb{R}^{k+1} \rightarrow \mathbb{R}$,  inductively by setting for $\lambda_0, \lambda_1 \in \mathbb{R}$, $\lambda \in \mathbb{R}^{k-1}$, 
\[
f^{[k]}(\lambda_0, \lambda_1,  \lambda) = \left\{
\begin{array}{ll}
\frac{f^{[k-1]}(\lambda_0, \lambda) - f^{[k-1]}(\lambda_1, \lambda)  }{\lambda_0 - \lambda_1} & \lambda_0 \not = \lambda_1, \\
 \frac{d}{d\mu} f^{[k-1]}(\mu, \lambda)  & \lambda_0   = \lambda_1. \\
\end{array}
\right.
\]
\end{definition} 
Note that in Section \ref{Sect=LowerBounds} we will extend the definition of $f^{[n]}$ to the function $f(s) = \vert s \vert s^{n-1}$ by setting $f^{[n]}(\lambda, \ldots, \lambda) = 0$, $\lambda \in \mathbb{R}$, and  defining $f^{[n]}$ as in Definition \ref{Dfn=DivDiff} in any other point.

Divided differences are invariant under permutation of the variables, meaning that for any permutation $\sigma \in S_{n+1}$ we have 
\[
f^{[n]}(\lambda_0, \ldots, \lambda_n) = f^{[n]}(\lambda_{\sigma(0)}, \ldots, \lambda_{\sigma(n)}).  
\]
Further, we have the following reduction formula, which can be found in \cite[Lemma 3.1]{CaspersReimann}. Very similar statements appeared already in  \cite[Eqn.\ (4.3) or (4.5)]{PSS-Inventiones}. 
    For all $\xi, \lambda_i, \lambda_j \in\mathbb{R}$ such that $\lambda_i\neq\lambda_j$, we have  
\begin{equation}\label{eqn: insert_one_xi}
\begin{split}
        f^{[n]}(\lambda_0,\dotsc,\lambda_n) = &\frac{\lambda_i-\xi}{\lambda_i-\lambda_j}f^{[n]}(\lambda_0, \dotsc,\lambda_i,\dotsc, \widehat{\lambda_j},\dotsc,\lambda_n, \xi) \\
        & + \frac{\xi-\lambda_j}{\lambda_i-\lambda_j}f^{[n]}(\lambda_0, \dotsc,\widehat{\lambda_i},\dotsc,{\lambda_j},\dotsc, \lambda_n, \xi). 
\end{split}
\end{equation}
This formula for $\xi =0$ for the function~\eqref{Eqn=AbsValue} was also obtained in \cite{CSZ-Israel}.

In our formulas, most divided differences will have repeated variables as input. We therefore introduce the following notation. Let $\alpha=(\alpha_0,\dotsc,\alpha_\kappa)\in \NNZero^{\kappa+1}$ of length 
\[
|\alpha|=\alpha_0+\dotsc+\alpha_\kappa= n+1,
\]
and let $(\lambda_0,\dotsc,\lambda_\kappa)\in\mathbb{R}^{\kappa+1}$. Then divided differences with repeated variables are denoted by \begin{equation*}
    f^{[n]}(\lambda_0^{(\alpha_0)},\dotsc,\lambda_\kappa^{(\alpha_\kappa)}):=f^{[n]}(\underbrace{\lambda_0,\dotsc,\lambda_0}_{\alpha_0\text{-times}},\dotsc,\underbrace{\lambda_\kappa,\dotsc,\lambda_\kappa}_{\alpha_\kappa\text{-times}}).
\end{equation*}
As the order of the divided difference can be inferred from the length of the multi-index $\alpha$, we will suppress it from the notation and write as shorthand 
\begin{equation}\label{Eqn=fBrackets}
f[\lambda^{(\alpha)}] := f[\lambda_0^{(\alpha_0)}, \ldots, \lambda_\kappa^{(\alpha_\kappa)}] := f^{[n]}(\lambda_0^{(\alpha_0)},\dotsc,\lambda_\kappa^{(\alpha_\kappa)}).
\end{equation}
By mild abuse of notation we shall regularly view the expression $f[\lambda^{(\alpha)}]$ as a function of $\lambda \in \mathbb{R}^{\kappa+1}$ given by $\lambda \mapsto f[\lambda^{(\alpha)}]$.

\subsection{Multilinear maps} Let $X_1, \ldots, X_n, X$ be normed vector spaces. Let $T: X_1 \times \ldots \times X_n \rightarrow X$ be a multilinear map, i.e. linear in each of the coordinates. The norm of $T$ is defined as 
\[
\Vert T: X_1 \times \ldots \times X_n \rightarrow X \Vert  := \sup_{x_i \in X_i, \Vert x_i \Vert_{X_i} = 1}  \Vert T(x_1, \ldots, x_n) \Vert_X.
\]

\subsection{Schatten classes}
For a Hilbert space $H$ we let $B(H)$ denote the bounded operators on $H$ with trace $\Tr$. For $p \in [1, \infty)$ we set the \emph{Schatten class} to be
\[
S_p := S_p(H) := \{ x \in B(H) \mid \Vert x \Vert_p := \Tr(\vert x \vert^p)^{\frac{1}{p}} < \infty \}. 
\]
Let $S_\infty$ be the space of compact operators with the operator norm. The spaces $(S_p, \Vert \: \cdot \: \Vert_p)$ are Banach spaces for all $p\in[1,\infty]$. 
The H\"older inequality states that the product 
\begin{equation}\label{Eqn=Mult}
{\rm Mult}: S_{p_1} \times \ldots \times S_{p_n} \rightarrow S_p: 
(x_1, \ldots, x_n) \mapsto x_1 \ldots x_n,
\end{equation}
is a contractive multilinear map for  $1 \leq p, p_1, \ldots, p_n \leq \infty$ with $p = ( \sum_{i=1}^n \frac{1}{p_i} )^{-1}$.

\subsection{Schur multipliers}
We may identify  $L^{2 }(\mathbb{R} \times \mathbb{R})$ with $ S_2(L^2(\mathbb{R}))$ by letting $A \in L^{2}(\mathbb{R} \times \mathbb{R})$ correspond to the operator $x_A$ determined by  
\[
\langle {x_A} \xi, \eta \rangle = \int_{\mathbb{R} \times \mathbb{R}} A(t, s) \xi(s) \overline{\eta}(t) dt ds, \quad \xi, \eta \in L^2(\mathbb{R}).
\]
This correspondence is linear and isometric and we call $A$
the {\it kernel} of $x_A$. 
 By mild abuse of notation, we will use $A$ to refer to both the integral operator $x_A$ and its kernel in the following definition.

\begin{definition}
Let $\phi: \mathbb{R}^{n+1} \rightarrow \mathbb{C}$ be a bounded Borel function. By \cite[Proposition 5]{CLS-AIF}, as a repeated application of the Cauchy-Schwartz inequality,  there exists a unique bounded map
\[
T_\phi: S_{2} \times \ldots \times S_2 \rightarrow S_2: (A_1, \ldots, A_n) \mapsto T_\phi(A_1, \ldots, A_n),
\]
 with norm $\Vert \phi \Vert_\infty$, such that the kernel of the operator $T_\phi(A_1, \ldots, A_n)$ is given by 
\[
T_\phi(A_1, \ldots, A_n)(s_0, s_n) = \int_{\mathbb{R}^{n-1}} \phi(s_0, \ldots, s_n) A_1(s_0, s_1) \ldots A_n(s_{n-1}, s_n) ds_1 \ldots ds_{n-1}, \quad s_0, s_n \in \mathbb{R},
\]
where we integrate over the kernels of $A_i$.  For $1 < p, p_1, \ldots, p_n < \infty$ with $p = ( \sum_{i=1}^n \frac{1}{p_i} )^{-1}$  we use the shorthand notation 
 \[
 \Vert \phi \Vert_{\frakm_{p_1, \ldots, p_n}} = \Vert T_\phi: S_{p_1} \times \ldots \times S_{p_n} \rightarrow S_p \Vert. 
 \]
\end{definition}

\begin{remark}\label{Rmk=ZeroSet}
By definition, multilinear Schur multipliers  only depend on the $L^{\infty}$-equivalence class of the symbol. More precisely, let $\phi, \psi: \mathbb{R}^{n+1} \rightarrow \mathbb{C}$ be bounded Borel functions such that $\{ \lambda \in \mathbb{R}^{n+1} \mid \phi(\lambda) \not = \psi(\lambda) \}$ has Lebesgue measure zero. Then $T_\phi = T_\psi$.  
\end{remark}

\subsection{The H\"ormander-Mikhlin-Schur multiplier theorem} The following theorem was proved by  Conde-Alonso, Gonz\'alez-P\'erez, Parcet and Tablate in \cite[Theorem A]{CGPT}; we state their theorem only for $d=1$. 

\begin{theorem}\label{Thm=HMS}
Let $\phi \in C^{ 1 }(\mathbb{R}^{2} \backslash \{ \lambda = \mu \} )$, $p \in (1, \infty)$. Then
\[
\Vert \phi \Vert_{\frakm_p} \lesssim  p p^\ast \TripleVert \phi \TripleVert_{ {\rm HMS}},
\]
where
\[
\TripleVert \phi \TripleVert_{ {\rm HMS}} = \max(  \Vert \phi \Vert_\infty ,  \Vert (\lambda, \mu) \mapsto     \vert \lambda - \mu \vert  ( \vert   \partial_\lambda \phi(\lambda, \mu) \vert + \vert   \partial_\mu \phi(\lambda, \mu) \vert    ) \Vert_\infty ).
\]
\end{theorem}

\begin{remark}\label{Rmk=Leibniz}
Let $\phi, \psi \in  C^{ 1 }(\mathbb{R}^{2} \backslash \{ \lambda = \mu \} )$. By the Leibniz rule we find 
\[
\TripleVert \phi \psi \TripleVert_{ {\rm HMS}} \leq   \Vert \phi \Vert_\infty  \TripleVert \psi  \TripleVert_{ {\rm HMS}} +   \Vert \psi \Vert_\infty \TripleVert \phi  \TripleVert_{ {\rm HMS}}. 
\]
\end{remark}

For the following theorem we use the notation introduced in \eqref{Eqn=fBrackets} and its subsequent discussion.  

\begin{theorem}\label{Thm=HMSConditions}
Let $\xi(\lambda) = \lambda_1 - \lambda_2$, where  $\lambda = (\lambda_1, \lambda_2)  \in \mathbb{R}^2$.  
We have for every $s \in \mathbb{R}\backslash \{ 0 \}$ and every $\alpha \in \NNOne^2$ with $\vert \alpha \vert = n+1$ that
\[
\begin{split}
\Vert \vert \xi \vert^{is} \Vert_{\HMS} \lesssim &  \vert s \vert, \\
\Vert {\rm sign}(\xi) \vert \xi \vert^{is} \Vert_{\HMS} \lesssim &  \vert s \vert, \\
\Vert   f[ \lambda^{(\alpha)} ]   \Vert_{\HMS} \lesssim &  \Vert f^{(n)} \Vert_\infty, \\
\Vert  \vert \xi \vert^{is} f[ \lambda^{(\alpha)} ]   \Vert_{\HMS} \lesssim & \vert s \vert \Vert f^{(n)} \Vert_\infty, \\
\Vert {\rm sign}(\xi)  \vert \xi \vert^{is} f[ \lambda^{(\alpha)} ]   \Vert_{\HMS} \lesssim & \vert s \vert \Vert f^{(n)} \Vert_\infty. \\
\end{split}
\]
\end{theorem}
\begin{proof}
The first and second inequality can be verified directly. The third inequality  holds by the same proof as \cite[Lemma 4.3]{CaspersReimann}. Note in particular that also $\Vert   f^{[n]}[ \lambda^{(\alpha)} ]   \Vert_{\infty} \lesssim   \Vert f^{(n)} \Vert_\infty$.  The fourth and fifth inequality then follows from the Leibniz rule of Remark \ref{Rmk=Leibniz}.  
\end{proof}

\section{A reduction formula for divided differences}\label{Sect=ReductionFormula}

The  formula \eqref{eqn: insert_one_xi} allows for a reduction on the number of input variables of a higher order divided difference. In this section we show that this formula can be extended to reduce the number of input variables even further. This leads to Corollary \ref{Cor=AlgebraicDecomposition}.

The following lemma is in the same spirit as~\cite[Lemma~5.8]{PSS-Inventiones}, however we are able to give an algebraic proof by specialising to divided differences {and at the same time relax the conditions imposed on $\xi$}.
\begin{lemma}\label{lemma: reduction formula divdiff} Let $f\in C^{n}(\mathbb{R})$, $n\ge 2$, and let $\xi\in\mathbb{R}$. Then for all $2\le \kappa\le n$, all multi-indices $\alpha=(\alpha_0,\dotsc,\alpha_\kappa)\in\NNZero^{\kappa+1}$ of length $|\alpha|=\alpha_0+\dotsc+\alpha_\kappa\le n+1$, and all $(\lambda_0,\dotsc,\lambda_\kappa)\in\mathbb{R}^{\kappa+1}$ such that $\lambda_i\neq \lambda_j$ and $\alpha_i,\alpha_j>0$ for some $i\neq j$,
    \begin{align*}
        f[\lambda_i^{(\alpha_i)}, \lambda_j^{(\alpha_j)},\widetilde{\lambda}^{(\widetilde{\alpha})}]= &\sum_{l=0}^{\alpha_i-1}\binom{\alpha_j+l-1}{l}\left(\frac{\lambda_i-\xi}{\lambda_i-\lambda_j}\right)^{\alpha_j}\left(\frac{\xi-\lambda_j}{\lambda_i-\lambda_j}\right)^{l}f[\lambda_i^{(\alpha_i-l)}, \xi^{(\alpha_j+l)},\widetilde{\lambda}^{(\widetilde{\alpha})}] \\ +  &\sum_{l=0}^{\alpha_j-1}\binom{\alpha_i+l-1}{l}\left(\frac{\lambda_i-\xi}{\lambda_i-\lambda_j}\right)^{l}\left(\frac{\xi-\lambda_j}{\lambda_i-\lambda_j}\right)^{\alpha_i}f[\xi^{(\alpha_i+l)},\lambda_j^{(\alpha_j-l)},\widetilde{\lambda}^{(\widetilde{\alpha})}],\numberthis\label{eqn: general_decomp_divdiff}
    \end{align*}
    where $\widetilde{\lambda}^{(\widetilde{\alpha})}:=(\lambda_0^{(\alpha_0)},\dotsc,\widehat{\lambda_i^{(\alpha_i)}},\dotsc,\widehat{\lambda_j^{(\alpha_j)}},\dotsc,{\lambda_\kappa^{(\alpha_\kappa)}})$ denotes the remaining variables with multiplicities. Here, we use the hat-notation introduced in Section~\ref{subsect: notation}.
\end{lemma}
By choosing $\xi=\lambda_{k}$ for some $k\in\{0,\dotsc,\kappa\}$ and a given tuple $(\lambda_0,\dotsc,\lambda_\kappa)\in\mathbb{R}^{\kappa+1}$,  we obtain the following reduction formula for divided differences of repeated variables. It allows us to express a divided difference in $\kappa+1$ variables as a linear combination of divided differences of the same order in $\kappa$ variables.  We let the polynomials in~\eqref{eqn: general_decomp_divdiff} be denoted by
\begin{equation}\label{eqn: decomp_polynomials}
p_{\alpha,l}(\xi) := \binom{\alpha + l -1}{ l }   \xi^{\alpha}(1-\xi)^l, \qquad \alpha, l \in \mathbb{N}_{\geq 0}.
\end{equation}

\begin{corollary}\label{Cor=AlgebraicDecomposition}
    In the setting of Lemma~\ref{lemma: reduction formula divdiff}, we have
    \begin{align*}
        f[\lambda_i^{(\alpha_i)}, \lambda_{k}^{(\alpha_{k})},\lambda_j^{(\alpha_j)},\widetilde{\lambda}^{(\widetilde{\alpha})}]= &\sum_{l=0}^{\alpha_i-1}
        p_{\alpha_j, l}\left( \frac{\lambda_i-\lambda_{k}}{\lambda_i-\lambda_j}    \right) f[\lambda_i^{(\alpha_i-l)}, \lambda_{k}^{(\alpha_{k}+\alpha_j+l)},\widetilde{\lambda}^{(\widetilde{\alpha})}] \\ +  &\sum_{l=0}^{\alpha_j-1}  p_{\alpha_i, l}\left( \frac{\lambda_k-\lambda_{j}}{\lambda_i-\lambda_j}    \right) f[\lambda_{k}^{(\alpha_{k}+\alpha_i+l)},\lambda_j^{(\alpha_j-l)},\widetilde{\lambda}^{(\widetilde{\alpha})}],\numberthis \\
        \widetilde{\lambda}^{(\widetilde{\alpha})}=&(\lambda_0^{(\alpha_0)},\dotsc,\widehat{\lambda_i^{(\alpha_i)}},\dotsc,\widehat{\lambda_j^{(\alpha_j)}},\dotsc,\widehat{\lambda_k^{(\alpha_k)}},\dotsc{\lambda_\kappa^{(\alpha_\kappa)}}),
    \end{align*}
    for any $k\in\{0,\dotsc,\kappa\}$.

\end{corollary}
\begin{proof}[Proof of Lemma~\ref{lemma: reduction formula divdiff}]
    As divided differences are invariant under permutation of variables, we may without loss of generality assume that $i=0$ and $j=1$. We prove the lemma by induction on $\alpha_0$ and $\alpha_1$.

    If $\alpha_0=\alpha_1=1$, we immediately obtain from~\eqref{eqn: insert_one_xi} that \begin{equation*}
        f[\lambda_0^{(1)},\lambda_1^{(1)},\widetilde{\lambda}^{(\widetilde{\alpha})}] = \frac{\lambda_0-\xi}{\lambda_0-\lambda_1}f[\lambda_0^{(1)}, \xi^{(1)},\widetilde{\lambda}^{(\widetilde{\alpha})}]+\frac{\xi-\lambda_1}{\lambda_0-\lambda_1}f[\xi^{(1)},\lambda_1^{(1)},\widetilde{\lambda}^{(\widetilde{\alpha})}].
    \end{equation*}

    Now assume that~\eqref{eqn: general_decomp_divdiff} holds for $\alpha_1=1$ and some   
  { $1<\alpha_0< n - \vert \widetilde{\alpha} \vert$.  }
    Applying~\eqref{eqn: insert_one_xi} and the induction hypothesis yields \begin{align*}
        f[\lambda_0^{(\alpha_0+1)}, \lambda_1^{(1)},\widetilde{\lambda}^{(\widetilde{\alpha})}] &= \frac{\lambda_0-\xi}{\lambda_0-\lambda_1}f[\lambda_0^{(\alpha_0+1)}, \xi^{(1)},\widetilde{\lambda}^{(\widetilde{\alpha})}]+\frac{\xi-\lambda_1}{\lambda_0-\lambda_1}f[\lambda_0^{(\alpha_0)},\xi^{(1)},\lambda_1^{(1)},\widetilde{\lambda}^{(\widetilde{\alpha})}] \\
        &= \frac{\lambda_0-\xi}{\lambda_0-\lambda_1}f[\lambda_0^{(\alpha_0+1)}, \xi^{(1)},\widetilde{\lambda}^{(\widetilde{\alpha})}]  \\&\quad+\frac{\xi-\lambda_1}{\lambda_0-\lambda_1}\bigg(\sum_{l=0}^{\alpha_0-1}\underbrace{\binom{1+l-1}{l}}_{=1\text{ for all }l}\frac{\lambda_0-\xi}{\lambda_0-\lambda_1}\left(\frac{\xi-\lambda_1}{\lambda_0-\lambda_1}\right)^{l}f[\lambda_0^{(\alpha_0-l)}, \xi^{(l+2)},\widetilde{\lambda}^{(\widetilde{\alpha})}]\big. \\   &\qquad\qquad\qquad \quad\big.+\underbrace{\binom{\alpha_0+0-1}{0}}_{=1}\left(\frac{\xi-\lambda_1}{\lambda_0-\lambda_1}\right)^{\alpha_0}f[\xi^{({\alpha_0+1})},\lambda_1^{(1)},\widetilde{\lambda}^{(\widetilde{\alpha})}]\bigg).
    \end{align*}
    Noting that all binomial coefficients are trivial for $\alpha_1=1$, we obtain our formula by rewriting the previous line as \begin{align*}&  \frac{\lambda_0-\xi}{\lambda_0-\lambda_1}f[\lambda_0^{(\alpha_0+1)}, \xi^{(1)},\widetilde{\lambda}^{(\widetilde{\alpha})}]  +\sum_{l=0}^{\alpha_0-1}\frac{\lambda_0-\xi}{\lambda_0-\lambda_1}\left(\frac{\xi-\lambda_1}{\lambda_0-\lambda_1}\right)^{l+1}f[\lambda_0^{(\alpha_0-l)}, \xi^{(  l+2 )},\widetilde{\lambda}^{(\widetilde{\alpha})}]\\
    & \qquad+ \left(\frac{\xi-\lambda_1}{\lambda_0-\lambda_1}\right)^{\alpha_0+1}f[\xi^{( \alpha_0 +1 )},\lambda_1^{(1)},\widetilde{\lambda}^{(\widetilde{\alpha})}] \\
    & = \sum_{l=0}^{\alpha_0}\frac{\lambda_0-\xi}{\lambda_0-\lambda_1}\left(\frac{\xi-\lambda_1}{\lambda_0-\lambda_1}\right)^{l}f[\lambda_0^{(\alpha_0+1-l)}, \xi^{( l+1)},\widetilde{\lambda}^{(\widetilde{\alpha})}] + \left(\frac{\xi-\lambda_1}{\lambda_0-\lambda_1}\right)^{\alpha_0+1}f[\xi^{( \alpha_0 +1)},\lambda_1^{(1)},\widetilde{\lambda}^{(\widetilde{\alpha})}],
    \end{align*}
    which finishes the proof in the case $\alpha_1=1$.

Now let $\alpha_1>1$. If $\alpha_0=1$, the statement follows from the previous case by invariance of the divided differences under permutation of variables. For the induction step, let $\alpha_1,\alpha_0>1$. As in the previous case, we first apply~\eqref{eqn: insert_one_xi} and obtain
\begin{align*}
    f[\lambda_0^{(\alpha_0)}, \lambda_1^{(\alpha_1)},\widetilde{\lambda}^{(\widetilde{\alpha})}]= \frac{\lambda_0-\xi}{\lambda_0-\lambda_1}f[\lambda_0^{(\alpha_0)}, \xi^{(1)},\lambda_1^{(\alpha_1-1)},\widetilde{\lambda}^{(\widetilde{\alpha})}]+\frac{\xi-\lambda_1}{\lambda_0-\lambda_1}f[\lambda_0^{(\alpha_0-1)},\xi^{(1)},\lambda_1^{(\alpha_1)},\widetilde{\lambda}^{(\widetilde{\alpha})}].
\end{align*}
We can now use the induction hypothesis to rewrite both divided differences as \begin{align*}
    &\quad\frac{\lambda_0-\xi}{\lambda_0-\lambda_1}\left(\sum_{l=0}^{\alpha_0-1}\binom{(\alpha_1-1)+l-1}{l}\left(\frac{\lambda_0-\xi}{\lambda_0-\lambda_1}\right)^{\alpha_1-1}\left(\frac{\xi-\lambda_1}{\lambda_0-\lambda_1}\right)^{l}f[\lambda_0^{(\alpha_0-l)}, \xi^{((\alpha_1-1)+l)},\widetilde{\lambda}^{(\widetilde{\alpha})}] \right.\\   &\qquad\qquad\quad+\left.\sum_{l=0}^{\alpha_1-2}\binom{\alpha_0+l-1}{l}\left(\frac{\lambda_0-\xi}{\lambda_0-\lambda_1}\right)^{l}\left(\frac{\xi-\lambda_1}{\lambda_0-\lambda_1}\right)^{\alpha_0}f[\xi^{(\alpha_0+l)},\lambda_1^{((\alpha_1-1)-l)},\widetilde{\lambda}^{(\widetilde{\alpha})}]\right) \\
    &+\frac{\xi-\lambda_1}{\lambda_0-\lambda_1}\left(\sum_{l=0}^{\alpha_0-2}\binom{\alpha_1+l-1}{l}\left(\frac{\lambda_0-\xi}{\lambda_0-\lambda_1}\right)^{\alpha_1}\left(\frac{\xi-\lambda_1}{\lambda_0-\lambda_1}\right)^{l}f[\lambda_0^{((\alpha_0-1)-l)}, \xi^{(\alpha_1+l)},\widetilde{\lambda}^{(\widetilde{\alpha})}] \right.\\   &\qquad\qquad\quad+\left.\sum_{l=0}^{\alpha_1-1}\binom{(\alpha_0-1)+l-1}{l}\left(\frac{\lambda_0-\xi}{\lambda_0-\lambda_1}\right)^{l}\left(\frac{\xi-\lambda_1}{\lambda_0-\lambda_1}\right)^{\alpha_0-1}f[\xi^{((\alpha_0-1)+l)},\lambda_1^{(\alpha_1-l)},\widetilde{\lambda}^{(\widetilde{\alpha})}]\right). \\
\end{align*}
By pulling the outermost fractions into the sum and renaming the summation index in the second and third sum, we may rewrite this expression as \begin{align*}
    &\sum_{l=0}^{\alpha_0-1}\binom{(\alpha_1-1)+l-1}{l}\left(\frac{\lambda_0-\xi}{\lambda_0-\lambda_1}\right)^{\alpha_1}\left(\frac{\xi-\lambda_1}{\lambda_0-\lambda_1}\right)^{l}f[\lambda_0^{(\alpha_0-l)}, \xi^{((\alpha_1-1)+l)},\widetilde{\lambda}^{(\widetilde{\alpha})}] \\   &\quad+\sum_{l=1}^{\alpha_1-1}\binom{(\alpha_0-1)+l-1}{l-1}\left(\frac{\lambda_0-\xi}{\lambda_0-\lambda_1}\right)^{l}\left(\frac{\xi-\lambda_1}{\lambda_0-\lambda_1}\right)^{\alpha_0}f[\xi^{((\alpha_0-1)+l)},\lambda_1^{(\alpha_1-l)},\widetilde{\lambda}^{(\widetilde{\alpha})}] \\
    &+\sum_{l=1}^{\alpha_0-1}\binom{(\alpha_1-1)+l-1}{l-1}\left(\frac{\lambda_0-\xi}{\lambda_0-\lambda_1}\right)^{\alpha_1}\left(\frac{\xi-\lambda_1}{\lambda_0-\lambda_1}\right)^{l}f[\lambda_0^{(\alpha_0-l)}, \xi^{((\alpha_1-1)+l)},\widetilde{\lambda}^{(\widetilde{\alpha})}]\\   &\quad+\sum_{l=0}^{\alpha_1-1}\binom{(\alpha_0-1)+l-1}{l}\left(\frac{\lambda_0-\xi}{\lambda_0-\lambda_1}\right)^{l}\left(\frac{\xi-\lambda_1}{\lambda_0-\lambda_1}\right)^{\alpha_0}f[\xi^{((\alpha_0-1)+l)},\lambda_1^{(\alpha_1-l)},\widetilde{\lambda}^{(\widetilde{\alpha})}]. \\
\end{align*}
We collect the summands with the same divided difference expression and obtain the following expression.
\begin{align*}
    &\left(\frac{\lambda_0-\xi}{\lambda_0-\lambda_1}\right)^{\alpha_1}f[\lambda_0^{(\alpha_0)}, \xi^{(\alpha_1-1)},\widetilde{\lambda}^{(\widetilde{\alpha})}]\\
    &+\sum_{l=1}^{\alpha_0-1}\left(\binom{(\alpha_1-1)+l-1}{l}+\binom{(\alpha_1-1)+l-1}{l-1}\right)\left(\frac{\lambda_0-\xi}{\lambda_0-\lambda_1}\right)^{\alpha_1} \\
    & \qquad \times \qquad \left(\frac{\xi-\lambda_1}{\lambda_0-\lambda_1}\right)^{l}f[\lambda_0^{(\alpha_0-l)}, \xi^{((\alpha_1-1)+l)},\widetilde{\lambda}^{(\widetilde{\alpha})}] \\ 
    &+ \sum_{l=1}^{\alpha_1-1}\left(\binom{(\alpha_0-1)+l-1}{l}+\binom{(\alpha_0-1)+l-1}{l-1}\right)\\
    & \qquad \times \qquad \left(\frac{\lambda_0-\xi}{\lambda_0-\lambda_1}\right)^{l}\left(\frac{\xi-\lambda_1}{\lambda_0-\lambda_1}\right)^{\alpha_0}f[\xi^{((\alpha_0-1)+l)},\lambda_1^{(\alpha_1-l)},\widetilde{\lambda}^{(\widetilde{\alpha})}] 
    \\
    &+\left(\frac{\xi-\lambda_1}{\lambda_0-\lambda_1}\right)^{\alpha_0}f[\xi^{(\alpha_0-1)},\lambda_1^{(\alpha_1)},\widetilde{\lambda}^{(\widetilde{\alpha})}]. \\
\end{align*}
From Pascal's identity, it follows for $k=0,1$ that \begin{equation*}
    \binom{(\alpha_k-1)+l-1}{l}+\binom{(\alpha_k-1)+l-1}{l-1} = \binom{(\alpha_k-1)+l}{l} = \binom{\alpha_k+l-1}{l},
\end{equation*} 
finishing the proof.
\end{proof}
\begin{remark}
    Note that the statement above is indeed a special case of~\cite[Lemma~5.8]{PSS-Inventiones}. There, a polynomial integral momentum is defined as \begin{align*}
        \phi_{\kappa,h,p}(\lambda_0,\dotsc,\lambda_{\kappa})&:=\int_{R_{\kappa}}p(s_0,\dotsc,s_{\kappa-1}, 1-\sum_{j=0}^{\kappa-1}s_j)h(s_0\lambda_0+\dotsc+s_{\kappa-1}\lambda_{\kappa-1}+(1-\sum_{j=0}^{\kappa-1}s_j)\lambda_{\kappa})ds,\\
        R_{\kappa}&:=\{(s_0,\dotsc,s_{\kappa-1})\in\mathbb{R}^{\kappa}\mid\sum_{j=0}^{\kappa-1}s_j\le 1,\;s_j\ge0,\;0\le j\le {\kappa -1} \},
    \end{align*}
    where $\kappa\in\NNOne$, $h\in C_b(\mathbb{R})$, and $p$ is a polynomial in $\kappa+1$ variables. Consider the special case where $h=f^{(n)}$ for some $f\in C^{n}(\mathbb{R})$ with bounded derivatives, and $p(s_0,\dotsc,s_{\kappa})=s_0^{\alpha_0}\cdots s_{\kappa}^{\alpha_{\kappa}}$ is a monomial with $|\alpha|=\alpha_0+\dotsc+\alpha_{\kappa}= n-\kappa$. 
    Now \begin{align*}
        &\phi_{\kappa,f^{(n)},p}(\lambda_0,\dotsc,\lambda_{\kappa}) \\&=\int_{  R_{\kappa}}s_0^{\alpha_0}\cdots s_{\kappa-1}^{\alpha_{\kappa}-1}(1-\sum_{j=0}^{\kappa-1}s_j)^{\alpha_{\kappa}}f^{(n)}(s_0\lambda_0+\dotsc+ s_{\kappa-1}\lambda_{\kappa-1}+(1-\sum_{j=0}^{\kappa-1}s_j)\lambda_{\kappa}) ds \\
        &= \partial_{\lambda_0}^{\alpha_0}\cdots \partial_{\lambda_{\kappa}}^{\alpha_{\kappa}}\int_{  R_{\kappa}}f^{(n-|\alpha|)}(s_0\lambda_0+\dotsc+ s_{\kappa-1}\lambda_{\kappa-1}+(1-\sum_{j=0}^{\kappa-1}s_j)\lambda_{\kappa})ds \\
        &=\partial_{\lambda_0}^{\alpha_0}\cdots \partial_{\lambda_{\kappa}}^{\alpha_{\kappa}} \phi_{\kappa,f^{(n-|\alpha|)},1}(\lambda_0,\dotsc,\lambda_{\kappa}).
    \end{align*}  
    In~\cite[Lemma~5.1]{PSS-Inventiones}, it was shown that $\phi_{\kappa,f^{(\kappa)},1}=f^{[\kappa]}$. So by noting that $n-|\alpha|=\kappa$ and using~\cite[Lemma~4.2]{CaspersReimann} to calculate the derivatives of the divided difference, we have \begin{align*}
        \partial_{\lambda_0}^{\alpha_0}\cdots \partial_{\lambda_{\kappa}}^{\alpha_{\kappa}} \phi_{\kappa,f^{(n-|\alpha|)},1}(\lambda_0,\dotsc,\lambda_{\kappa}) &= \partial_{\lambda_0}^{\alpha_0}\cdots \partial_{\lambda_{\kappa}}^{\alpha_{\kappa}} f^{[n-|\alpha|]}(\lambda_0,\dotsc,\lambda_{\kappa}) \\
        &=\alpha_0!\cdots\alpha_{\kappa}!f^{[n]}(\lambda_0^{(\alpha_0+1)},\dotsc,\lambda_{\kappa}^{(\alpha_{\kappa}+1)}).
    \end{align*}
    More generally, for $|\alpha|\le n-\kappa$, we obtain \begin{equation*}
        \phi_{\kappa,f^{(n)},p}(\lambda_0,\dotsc,\lambda_{\kappa})=\alpha_0!\cdots\alpha_{\kappa}!(f^{(n-|\alpha|-\kappa)})^{[|\alpha|+\kappa]}(\lambda_0^{(\alpha_0+1)},\dotsc,\lambda_{\kappa}^{(\alpha_{\kappa}+1)}).
    \end{equation*}
\end{remark}

\section{Expansions for smooth homogeneous symbols}\label{Sect=Homogeneous} 
The main theorem of this section is a multivariable result of a decomposition that occurs in the proof of \cite[Theorem 5.1]{CaspersReimann} and whose origin lies in \cite[Lemma 6]{PoSu11}. An important technicality that we overcome is that the support of our function can intersect with the hyperplane given by  $\{ \xi \in \mathbb{R}^n \mid \xi_n = 0\}$; the function only needs to vanish on this hyperplane and its support should not intersect any hyperplane given by $\{ \xi \in \mathbb{R}^n \mid \xi_k = 0\}$, $k = 1, \ldots, n-1$.  Already in the $n=2$ case our result is therefore more general than \cite[Theorem 5.1]{CaspersReimann}. This also leads directly to an improvement of the asymptotic constants in our main estimate compared to \cite{PSS-Inventiones} as we avoid the use of triangular truncations, see \cite[p.\ 533, lines 4 to 6]{PSS-Inventiones}.

\begin{remark}
The support of a function $\phi: \mathbb{R}^n \backslash \{ 0 \} \rightarrow \mathbb{C}$ is denoted as $\supp(\phi)$ and is defined to be the closure of the set 
\[
\{ \xi \in \mathbb{R}^{n} \backslash \{ 0 \} \mid \phi(\xi) \not = 0 \}
\]
in $\mathbb{R}^n \backslash \{ 0 \}$ and in particular does not include the point $0\in\mathbb{R}^n$. If $\phi$ is homogeneous then the support is a homogeneous set. 
\end{remark}

\begin{proposition}\label{prop=exists_rho}
Let $n \geq 2$ and let $\phi: \mathbb{R}^n \backslash \{ 0 \}  \rightarrow \mathbb{C}$ be a continuous, even homogeneous function. Assume further that for every $1 \leq k \leq n-1$ we have  
\begin{equation}\label{Eqn=SupportAssumption}
\supp(\phi) \cap \{ \xi \in \mathbb{R}^n \backslash \{ 0 \}  \mid  \xi_k = 0  \} = \emptyset.
\end{equation}
Then the function on $\mathbb{R}^{n-1}$ given by 
\begin{equation}\label{Eqn=PsiConstruct}
\psi(s_1, \ldots, s_{n-1}) := \phi(1,  s_1, s_1 s_2, \ldots    , s_1 s_2 \ldots s_{n-2}, s_1 s_2 \ldots s_{n-1}),
\end{equation}
is continuous, compactly supported and  for $\xi \in \mathbb{R}^{n}$ with  $\xi_k \not = 0$ for all $1 \leq k \leq n-1$ we have
\begin{equation}\label{Eqn=HomogeneousFunctionPhiPsi}
\phi(\xi_1, \ldots, \xi_n) = \psi(\frac{\xi_2}{\xi_1}, \frac{\xi_3}{\xi_2}, \ldots, \frac{\xi_n}{\xi_{n-1}}).
\end{equation}
\end{proposition}
\begin{proof}
The equality \eqref{Eqn=HomogeneousFunctionPhiPsi} follows as $\phi$ is even homogeneous, as indeed
\[
\psi(\frac{\xi_2}{\xi_1}, \frac{\xi_3}{\xi_2}, \ldots, \frac{\xi_n}{\xi_{n-1}}) = \phi(1, \frac{\xi_2}{\xi_1}, \ldots, \frac{\xi_n}{\xi_1}) = \phi(\xi_1, \ldots, \xi_n). 
\]
Clearly $\psi$ is continuous by \eqref{Eqn=PsiConstruct}.
It remains to justify that $\psi$ has compact support. Since~$\phi$ satisfies~\eqref{Eqn=SupportAssumption} and is homogeneous, we have that for every $1 \leq k \leq n-1$ there exists some~$\varepsilon_k>0$ such that
    $\phi(\xi_1, \ldots, \xi_n) = 0$ if $|\xi_{k} |<\varepsilon_k \vert \xi_{k+1} \vert$. Let $\varepsilon = \min(\varepsilon_1, \ldots, \varepsilon_{n-1})$. For all~${(s_1,\dotsc,s_{n-1})\in\mathbb{R}^{n-1}}$ such that $|s_k|>\tfrac{1}{\varepsilon}$ for some $k=1,\dotsc,n-1$, we see by definition of~$\psi$  that 
    \begin{align*}
        \psi(s_1,\dotsc,s_{n-1})&= \phi(1,  s_1, s_1 s_2, \ldots    , s_1 s_2 \ldots s_{n-2}, s_1 s_2 \ldots s_{n-1}) = 0, 
    \end{align*}
   as indeed $\xi_{k} := s_1 \ldots s_{k-1}$ and $\xi_{k+1} := s_1 \ldots s_{k}$ have  the property that $\vert \xi_{k} \vert < \varepsilon \vert \xi_{k+1}\vert$. Hence in particular~${\vert \xi_{k} \vert < \varepsilon_k \vert \xi_{k+1}\vert}$, and thus     
 $\psi$ has compact support.

\end{proof}

\begin{definition}\label{Dfn=Hn}
Let $\mathcal{H}_n$, $n \geq 2$ be the set of smooth homogeneous functions~$\phi: \mathbb{R}^n \backslash \{ 0\} \rightarrow  \mathbb{C}$   such that  for all $\xi_1, \ldots, \xi_{n-1} \in \mathbb{R}$ we have  
\begin{equation}\label{Eqn=HnSuppConiditonA}
    \phi(\xi_1, \ldots, \xi_{n-1},  0) = 0,
\end{equation} 
and for all $1 \leq k \leq n-1$ we have 
\begin{equation}\label{Eqn=HnSuppConiditon}  
\supp(\phi) \cap \{ \xi \in \mathbb{R}^{n} \backslash \{ 0 \} \mid \xi_k = 0 \} = \emptyset. 
\end{equation}
\end{definition}

\begin{proposition}\label{Prop=FourierTypeDecomposition}
Let $n \geq 2$  and let $\phi \in \mathcal{H}_n$.  Then  exist Schwartz functions $g_{\epsilon}: \mathbb{R}^{n-1} \rightarrow \mathbb{C}$, where~${\epsilon \in \{ 0, 1\}^n}$, such that for every $\xi_1, \ldots, \xi_{n}  \in \mathbb{R}  \backslash \{ 0 \}$   we have
 \[
\begin{split}
\phi(\xi_1, \ldots,  \xi_n) =  & \!\!\!\!\!\!  \sum_{ \epsilon  \in \{ 0, 1 \}^n } \left( \prod_{i=1}^n {\rm sign}( \xi_i )^{\epsilon_i}  \right)   \int_{\mathbb{R}} g_{\epsilon}(s_1, \ldots,  s_{n-1}) \\
& \qquad \times \qquad  \vert \xi_1 \vert^{ -is_1} \vert \xi_2 \vert^{ i(s_1 - s_2)}  \ldots   \vert \xi_{n-1} \vert^{i(s_{n-2}-s_{n-1})}  \vert \xi_n \vert^{ is_{n-1}}    ds_1 \ldots  ds_{n-1}.  
\end{split}
\]
\end{proposition} 
\begin{proof}
We start by writing the symbol $\phi$ as a sum of symbols that are even or odd homogeneous with respect to each of its coordinates. To this extent let $\epsilon = (\epsilon_1, \ldots, \epsilon_n) \in \{ 0, 1 \}^n$ and set
\[
\phi_{\epsilon_1, \ldots, \epsilon_n}(\xi_1, \ldots, \xi_n) 
= \frac{1}{2^n} \sum_{\sigma_1, \ldots, \sigma_n \in \{ +1, -1 \} }  \left( \prod_{i=1}^n   \sigma_i^{\epsilon_i}  \right)    \phi( \sigma_1 \xi_1, \ldots, \sigma_n \xi_n ). 
\]
It then follows that 
\begin{equation}\label{Eqn=Average}
\begin{split}
\phi(\xi_1, \ldots,  \xi_n) =  & \!\!\!\!\!\!  \sum_{ \epsilon_1, \ldots, \epsilon_n \in \{ 0, 1 \}  }     
\phi_{\epsilon_1, \ldots, \epsilon_n}(\xi_1, \ldots, \xi_n).
\end{split}
\end{equation}
Note that $\phi_{\epsilon_1, \ldots, \epsilon_n}(\xi_1, \ldots, \xi_n)$ is even (respectively odd) homogeneous in the variable $\xi_k$ in case~${\epsilon_k = 0}$ (respectively $\epsilon_k = 1$). In particular, for $\xi_1, \ldots, \xi_n \in \mathbb{R} \backslash \{ 0 \}$, 
\begin{equation}\label{Eqn=Homogenity}
\phi_{\epsilon_1, \ldots, \epsilon_n}(\xi_1, \ldots, \xi_n) = \left( \prod_{i=1}^n {\rm sign}( \xi_i )^{\epsilon_i}  \right)  \phi_{\epsilon_1, \ldots, \epsilon_n}(\vert \xi_1 \vert, \ldots, \vert \xi_n \vert).
\end{equation}
In turn, we may apply Proposition~\ref{prop=exists_rho}  to the function 
\[
\xi \mapsto \phi_{\epsilon_1, \ldots, \epsilon_n}(\vert \xi_1 \vert, \ldots, \vert \xi_n \vert),
\]
which gives a continuous, compactly supported function  $\psi_{\epsilon_1, \ldots, \epsilon_n}: \mathbb{R}^{n-1} \rightarrow \mathbb{C}$ such that for all~$\xi_1, \ldots, \xi_n \in \mathbb{R} \backslash \{ 0 \}$ we have 
\[
\phi_{\epsilon_1, \ldots, \epsilon_n}(\xi_1, \ldots, \xi_n) =  \left( \prod_{k=1}^n {\rm sign}( \xi_k )^{\epsilon_k}  \right)   \psi_{\epsilon_1, \ldots, \epsilon_n}(\frac{ \vert \xi_2 \vert}{\vert \xi_1 \vert}, \frac{ \vert \xi_3 \vert}{ \vert \xi_2 \vert}, \ldots, \frac{ \vert\xi_n \vert}{ \vert\xi_{n-1} \vert}).
\] 
Consider the function
\[
\widehat{g}_{\epsilon_1, \ldots, \epsilon_n}: \mathbb{R}^{n-1} \rightarrow \mathbb{C}:  (t_1, \ldots, t_{n-1}) \mapsto  \psi_{\epsilon_1, \ldots, \epsilon_n}(e^{t_1}, \ldots, e^{t_{n-1}}).
\]
\noindent {\bf Claim:}    $\widehat{g}_{\epsilon_1, \ldots, \epsilon_n}$ is a Schwartz function.

\vspace{0.3cm}

\noindent {\it Proof of the claim:} 
As $\psi_{\epsilon_1, \ldots, \epsilon_n}$ has compact support, see Proposition \ref{prop=exists_rho},  the support of $\widehat{g}_{\epsilon_1, \ldots, \epsilon_n}$ is contained in $(-\infty, K]^{n-1}$ for some $K> 0$.

Recall that 
\[
\psi_{\epsilon_1, \ldots, \epsilon_n}(e^{t_1}, \ldots, e^{t_{n-1}}) = 
\phi_{\epsilon_1, \ldots, \epsilon_n}(1, e^{t_1}, e^{t_1+t_2},  \ldots, e^{t_1 + \ldots + t_{n-1}})
\]
by Proposition~\ref{prop=exists_rho}. As $\phi_{\epsilon_1, \ldots, \epsilon_n}$ is smooth on $\mathbb{R}^{n} \backslash \{ 0 \}$, we find that  $\widehat{g}_{\epsilon_1, \ldots, \epsilon_n}$ is smooth. 
Further, recall that $\phi_{\epsilon_1, \ldots, \epsilon_n}(\xi_1, \ldots, \xi_{n-1}, 0)=0$ by assumption. By the fundamental theorem of calculus and smoothness of $\phi_{\epsilon_1, \ldots, \epsilon_n}$, we may write for $\xi_i \geq 0$,
\begin{equation*}
    \phi_{\epsilon_1, \ldots, \epsilon_n}(\xi_1, \ldots, \xi_{n-1}, \xi_n) = \xi_n \int_0^1  (\partial_{\xi_n} \phi_{\epsilon_1, \ldots, \epsilon_n})( \xi_1 , \xi_2 \ldots, t \xi_n) dt.
\end{equation*}
As $\phi$ is smooth   we see that on any compact neighbourhood $D$ of $(1,0,\dotsc,0)\in\mathbb{R}^n$ not containing~0,  
the latter integral has absolute value bounded by some constant $C$. We may apply this to any such $D$ that contains the image of the set $(-\infty, K]^{n-1}$ under the map
\[
t \mapsto (1, e^{t_1}, e^{t_1+t_2},  \ldots, e^{t_1 + \ldots + t_{n-1}}). 
\]
Hence altogether,  we see that whenever $t_1, \ldots, t_{n-1} \leq K$ then we have 
\begin{equation*}
   | \psi_{\epsilon_1, \ldots, \epsilon_n}(e^{t_1}, \ldots, e^{t_{n-1}}) \vert  \leq  e^{t_1 + \ldots + t_{n-1}} C.
\end{equation*}
The right hand side function in the latter equation decays faster than any polynomial on $(-\infty, K]^{n-1}$.

\vspace{0.3cm}

\noindent {\it Remainder of the proof:}  It follows that the inverse Fourier transform $g_\epsilon := g_{\epsilon_1, \ldots, \epsilon_n}$ of $\widehat{g}_{\epsilon_1, \ldots, \epsilon_n}$  is Schwartz and we have
\[
 \psi_{\epsilon_1, \ldots, \epsilon_n}(e^{t_1}, \ldots, e^{t_{n-1}}) = \int_{\mathbb{R}^{n-1}}  g_{\epsilon_1, \ldots, \epsilon_n}(s_1, \ldots, s_{n-1}) e^{i (s_1 t_1 + \ldots + s_{n-1}t_{n-1})} ds_1 \ldots ds_{n-1}. 
\] 
We evaluate this formula for $t_1 = \log(\frac{\vert \xi_2 \vert}{\vert \xi_1 \vert} ), \ldots, t_{n-1}  = \log(\frac{\vert \xi_n \vert}{\vert \xi_{n-1} \vert} )$,  $\xi_1, \ldots, \xi_n \in \mathbb{R} \backslash \{ 0\}$, which gives
\begin{equation}\label{Eqn=FourierExpansion}
\begin{split}
\phi_{\epsilon_1, \ldots, \epsilon_n}(\vert \xi_1 \vert, \ldots, \vert \xi_n \vert)  = & \psi_{\epsilon_1, \ldots, \epsilon_n}(\frac{\vert \xi_2 \vert}{\vert \xi_1 \vert}, \ldots, \frac{\vert \xi_{n} \vert}{\vert \xi_{n-1} \vert})  \\
= &   \int_{\mathbb{R}^{n}} g_{\epsilon_1, \ldots, \epsilon_n}(s_1, \ldots,  s_{n-1}) \vert \xi_1 \vert^{-i s_1} \vert \xi_2 \vert^{i(s_1 - s_2)}  \ldots  \\
& \qquad \times \qquad \vert \xi_{n-1} \vert^{i (s_{n-2} - s_{n-1})}   \vert \xi_{n} \vert^{i s_{n-1} }   ds_1 \ldots ds_{n-1}.  
\end{split}
\end{equation}
The proposition now follows from \eqref{Eqn=Average}, \eqref{Eqn=Homogenity} and \eqref{Eqn=FourierExpansion}.  
\end{proof}

\section{Decomposition of divided differences}\label{Sect=DivDiffDecomposition}

{In this section we decompose divided differences as a linear combination of products of: (i) a smooth homogeneous function depending only on the coordinate differences and satisfying the conditions of Theorem \ref{Prop=FourierTypeDecomposition}, and (ii) the same divided difference but with some of the input variables repeating.  We do this inductively and reduce the number of input variables in   (ii) until only two variables remain. } 

As our main result is proved in the linear case we focus here on the multilinear case only. Therefore:

 \vspace{0.3cm}
 
 \noindent {\bf Convention.} Throughout this section fix $n \in \mathbb{N}$, $n \geq 2$. 

\subsection{Outline of the proof}
In spirit, our proof follows the structure as in~\cite{CaspersReimann}, and to a smaller degree also \cite{PSS-Inventiones}. However, both the fact that we treat the $n$-linear case and our aim for optimal constants require a significant number of extra technical and conceptual obstacles to overcome. This also leads to a fairly technical and notation-heavy construction, hence we give a formal overview of the argument here.
 
 \vspace{0.3cm}
 
We will repeatedly apply the reduction formula~\eqref{eqn: general_decomp_divdiff}. However, the fractions that appear as a byproduct of this reduction formula are unbounded as functions in ${(\lambda_0,\dotsc,\lambda_n)\in\mathbb{R}^{n+1}\setminus\{0\}}$, which we need to control by only considering them on certain domains.  Therefore, similar to~\cite[Section~3]{CaspersReimann}, we construct a smooth partition of unity of Toeplitz form to control the fractions appearing in~\eqref{eqn: general_decomp_divdiff}.  

  \vspace{0.3cm}
  
More precisely, note that each time we apply Lemma~\ref{lemma: reduction formula divdiff} to a divided difference 
\[
f^{[n]}[\lambda_{i_0}^{(\alpha_{i_0})},\dotsc,\lambda_{i_k}^{(\alpha_{i_k})}],
\]
we need to make a choice of indices $i,j\in\{i_0,\dotsc,i_k\}$, as well as of the $\xi\in\mathbb{R}$ which we insert into the divided difference. Here, we will always choose $i=i_{l-1}$, $j=i_l$ and $\xi=\lambda_{i_{l+1}}$,  with cyclic numbering of the indices, which reduces our choices to choosing some $l\in\{0,\dotsc,k\}$.  Moreover, note that each application of Lemma~\ref{lemma: reduction formula divdiff} to a divided difference in variables $I:=\{\lambda_{i_0},\dotsc,\lambda_{i_k}\}$ results in a collection of divided differences which take as input the variables in either one of the two sets $I_+:=I\setminus\{\lambda_j\}$ and $I_{-}:=I\setminus\{\lambda_i\}$. We repeatedly apply this procedure of selecting one of the indices $i_l$ and then one of the branches $\pm$ on which we continue our decomposition. 

 \vspace{0.3cm}
 
We encode each possible string of choices in a sequence $F=(i_1,\sigma_1,i_2,\sigma_2,\dots,i_{n-1},\sigma_{n-1})$, where each $i_l$ denotes our chosen index, and $\sigma_l=\pm$ encodes whether in the next step we will treat the obtained divided differences on the variables $I_+$ or $I_-$. Note that it is possible to arrive at the same set of input variables through different choices of indices in previous applications of the lemma. The advantage of our construction is that at each step $l$, the sequence $F$ tells us not only the set of input variables we are considering, but also the previous decomposition choices, which allows us to keep track of the fractions that accompany the divided differences.
 
\vspace{0.3cm}
 
Let $\mathcal{F}_{n,n-1}$ denote the collection of all such choice sequences $F$ as in the previous paragraph. In Section~\ref{subsec: partition_of_unity}, we construct a partition of unity $(U_F)_{F\in\mathcal{F}_{n,n-1}}$ adapted to the choice sequences in the following way.  Suppose   some $F\in\mathcal{F}_{n,n-1}$ is given.  If a product of a divided difference with two distinct variables with Toeplitz-form fractions is such that it can be obtained by repeatedly applying Lemma~\ref{lemma: reduction formula divdiff} to $f^{[n]}$, with the choices of the index $i_l$ at each step prescribed by $F$, then all fractions in this product will be bounded on $U_F$.
 
\vspace{0.3cm}

After establishing some notation in Section~\ref{subsec: sec_5_aux_function}, we prove in Section~\ref{Sect=Reduction}   that one can indeed decompose $f^{[n]}$ restricted to $U_F$ into  
a linear combination of products of homogeneous multivariable functions depending on coordinate differences (i.e.\ of Toeplitz form) and  two-variable $n$th order divided differences. We can do this in such a way that the homogeneous part satisfies the conditions of Proposition~\ref{prop=exists_rho} and in such a way that the variables of the homogeneous function match the variables of the divided difference (see  Remark \ref{Rmk=Matching}). Both these facts are crucial towards our proof of the sharp boundedness of the Schur multiplier with symbol $f^{[n]}$ in Section~\ref{sect: upper_bound_main_result}.

\subsection{Partitions of unity}\label{subsec: partition_of_unity}   For $2 \leq  k \leq n$ we consider  $\Sphere^{k-1}$ to be the $k-1$-sphere 
  \[
  \Sphere^{k-1} = \{ \xi \in \mathbb{R}^k \mid \Vert \xi \Vert_2 = 1\}.
  \]
We start by constructing  coverings of $\Sphere^{k-1}$ and  $\mathbb{R}^{k} \backslash \{ 0 \}$ as follows. Let  $\varepsilon > 0$ and $1 \leq l \leq k$.  Set  
\[
\begin{split}
\widetilde{U}_{k,l, \varepsilon} := & \{ \xi \in \Sphere^{k-1} \backslash \{ 0 \}  \mid  \forall \: 1 \leq  b \leq k:  \vert    \xi_{b} \vert < (1+ \varepsilon)  \vert    \xi_{l} \vert   \}, \\
\mathbb{R}_{> 0} \cdot  \widetilde{U}_{k,l, \varepsilon} := &   \{ \xi \in \mathbb{R}^{k} \backslash \{ 0 \}  \mid   \forall \: 1 \leq  b \leq k:  \vert    \xi_{b} \vert < (1+ \varepsilon)  \vert    \xi_{l} \vert      \}.
\end{split}
\] 
Let  $I \subseteq \{ 0, \ldots, n\}$ be a set with at least 3 elements. 
  We define the space of vectors in $\mathbb{R}^{n+1}$ whose coordinates indexed by $I$ are all the same as 
\[
\Delta_{I,n}  := \left\{ (\lambda_0, \ldots, \lambda_n) \in \mathbb{R}^{n+1} \mid \forall i,j \in I: \lambda_i = \lambda_j  \right\}.
\]
Now let $\vert I \vert=k+1$ and write   
\begin{equation}\label{Eqn=LabelsofI}
I = \{ i_0 < i_1 < \ldots < i_{k} \}.
\end{equation} 
For $i_l   \in I \backslash \{ i_0 \}$ and   $\varepsilon > 0$ we set 
\[
\begin{split}
  U_{I, i_l, \varepsilon} := & 
\left\{
(\lambda_0, \ldots, \lambda_n) \in  \mathbb{R}^{n+1} \backslash \Delta_{I,n}     \mid \forall \:  i_a \in I \backslash \{ i_0 \}:  \vert  \lambda_{i_a} -  \lambda_{i_{a-1}}  \vert < (1+ \varepsilon)     \vert \lambda_{i_l} - \lambda_{i_{l-1}} \vert
\right\}. 
\end{split}
\]
 Note that these sets are open subsets of $\mathbb{R}^{n+1} \backslash \Delta_{I,n}$. The map 
\begin{equation}\label{Eqn=QMapNew}
Q_I: \mathbb{R}^{n+1} \backslash \Delta_{I,n} \rightarrow \mathbb{R}^k \backslash \{ 0 \}: (\lambda_0, \ldots, \lambda_n) \mapsto (\lambda_{i_1} - \lambda_{i_0}, \ldots, \lambda_{i_k} - \lambda_{i_{k-1}}),
\end{equation}
is a surjection mapping $U_{I, i_l, \varepsilon}$ to  $\mathbb{R}_{>0} \cdot \widetilde{U}_{k,l, \varepsilon}$. In the case $\varepsilon = 1$, set 
\[
\widetilde{U}_{I,l} := \widetilde{U}_{I,l, 1}, \quad \mathbb{R}_{> 0} \cdot \widetilde{U}_{I,l} :=  \mathbb{R}_{> 0} \cdot \widetilde{U}_{I,l, 1}, \quad U_{I,i_l} :=  U_{I,i_l, 1}. 
\]
The proofs below work for every $\varepsilon > 0$ but after the next lemma we simply work with  $\varepsilon = 1$. This convention is comparable to the fact that \eqref{Eqn=TheoryOfEverything} is proved for $0<\lambda < (1+\varepsilon) \mu$ with $\varepsilon = 1$.

\begin{lemma}\label{Lem=CoversItSphere}
The following holds true. 
\begin{enumerate}
    \item \label{Item=Cover1} The sets $(\widetilde{U}_{k, l, \varepsilon})_{ 1 \leq l \leq k }$ cover $\Sphere^{k-1}$.
    \item \label{Item=Cover1b} The sets $(\mathbb{R}_{>0} \cdot \widetilde{U}_{k, l, \varepsilon})_{ 1 \leq l \leq k }$ cover $\mathbb{R}^{k} \backslash \{ 0 \}$. 
        \item \label{Item=Cover2} The sets $(U_{I, i_l, \varepsilon})_{ 1\leq l \leq k }$ cover $\mathbb{R}^{n+1} \backslash \Delta_{I,n}$.
\end{enumerate} 
\end{lemma}
\begin{proof}
\eqref{Item=Cover1} For $\xi \in \Sphere^{k-1}$ we let $1 \leq l \leq k$ be such that $\vert \xi_l \vert$ is largest. Then $\xi \in \widetilde{U}_{k, l, \varepsilon}$. The proof of \eqref{Item=Cover1b} is similar. \eqref{Item=Cover2}  follows from  \eqref{Item=Cover1b} and the fact that $Q_I$ given in  \eqref{Eqn=QMapNew} is a continuous surjection mapping $U_{I,i_l, \varepsilon}$ onto  $\mathbb{R}_{>0} \cdot \widetilde{U}_{k,l, \varepsilon}$. 
\end{proof}

The following figure gives a useful visualization of some of these sets.

\begin{figure}[ht]
    \centering
    \includegraphics[scale=0.18]{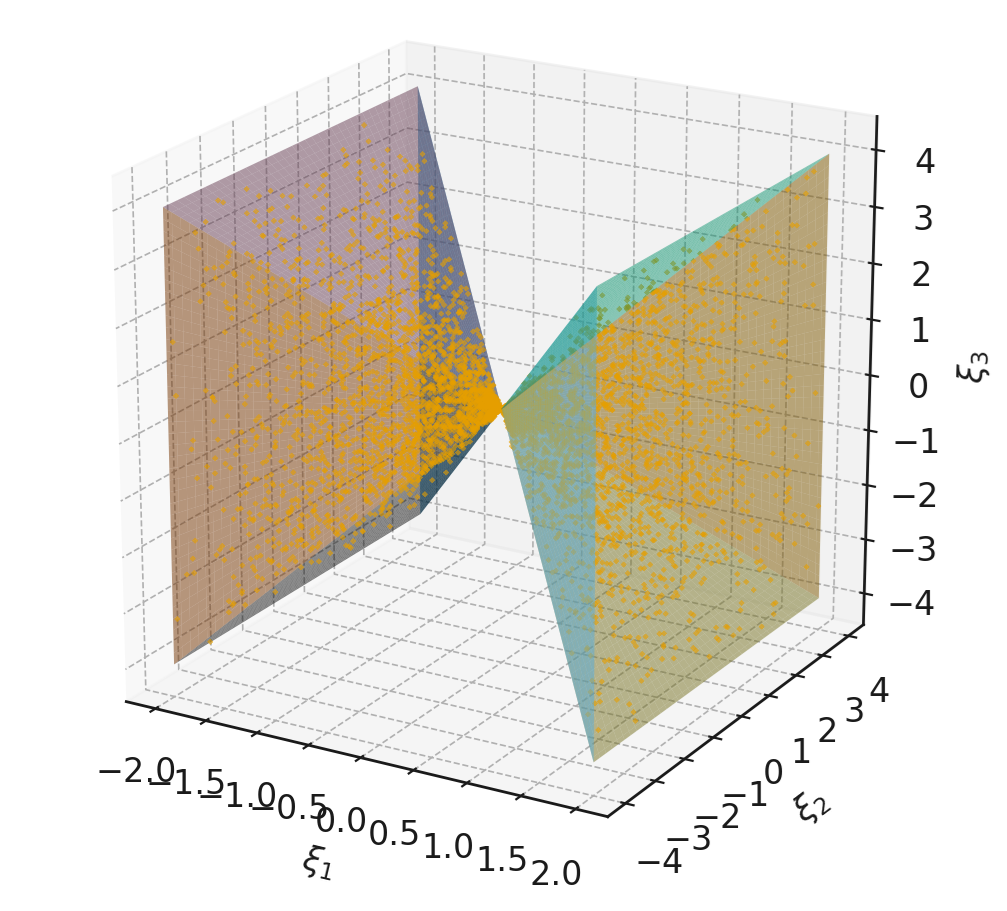}    \includegraphics[scale=0.13]{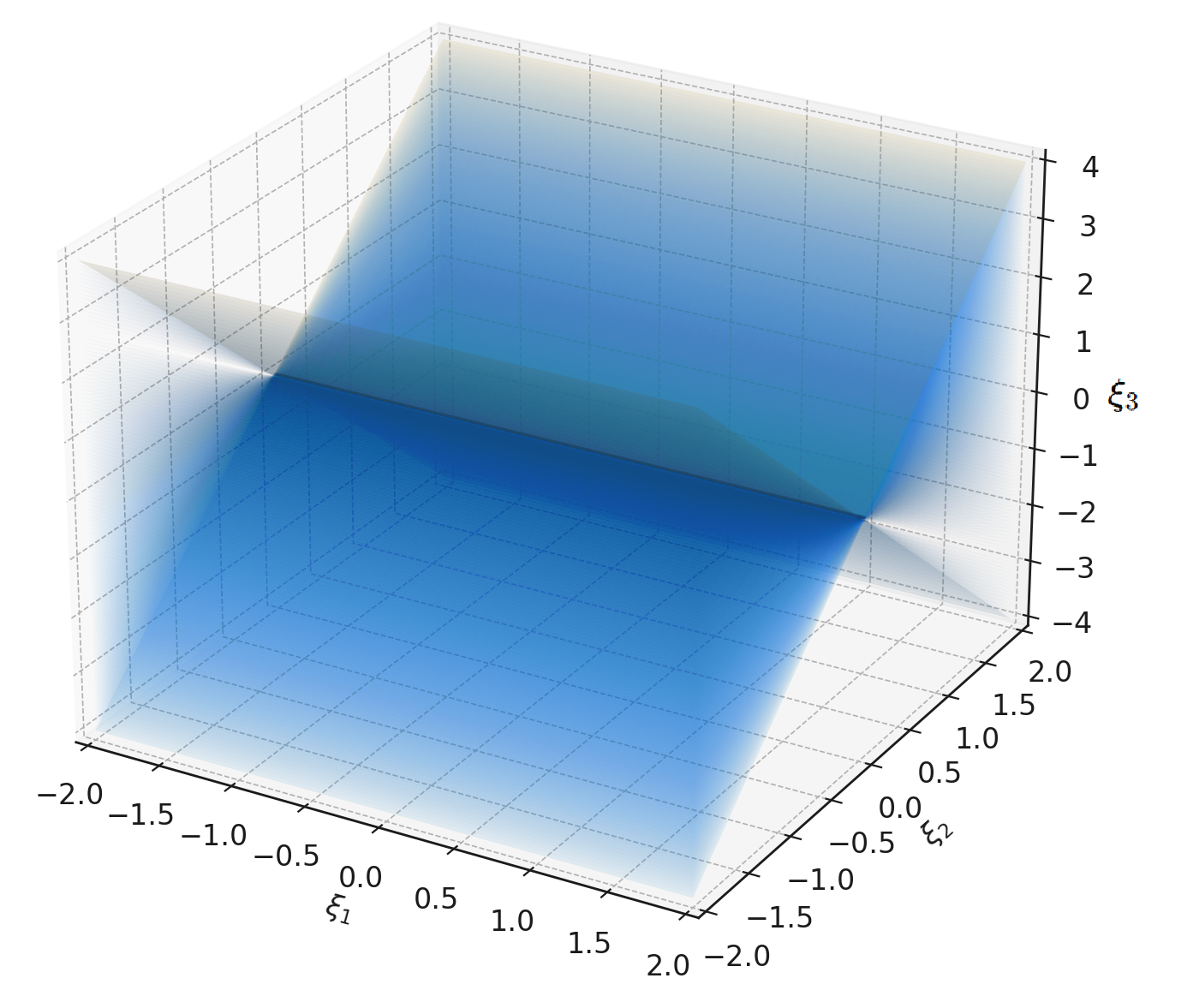} \includegraphics[scale=0.13]{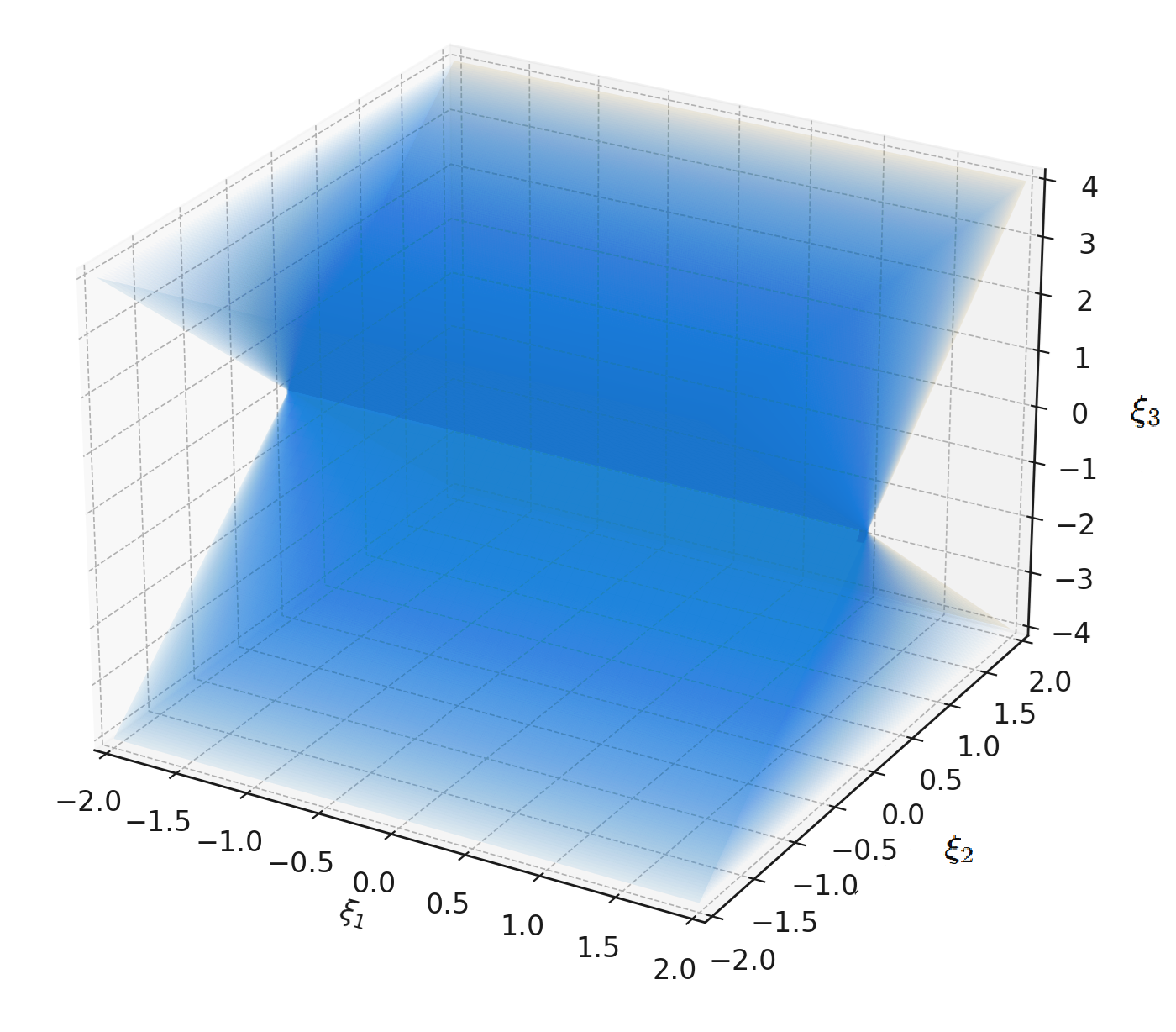}
    \caption{ {From left to right, pictures of the sets:
    \[
    \begin{split}
        V_1 & = \{ (\xi_1, \xi_2, \xi_3) \in \mathbb{R}^3 \mid \vert \xi_2 \vert < 2 \vert \xi_1 \vert,  \vert \xi_3 \vert < 2 \vert \xi_1 \vert \}, \\
        V_2 & =\{ (\xi_1, \xi_2, \xi_3) \in \mathbb{R}^3 \mid \vert \xi_3 \vert < 2 \vert \xi_2 \vert  \}, \\
        V_3 &  =\{ (\xi_1, \xi_2, \xi_3)  \in \mathbb{R}^3 \mid \vert \xi_2 \vert < 2 \vert \xi_3 \vert  \}.
        \end{split}
    \]
    We have 
    \[
    \begin{split}
    U_{\{0,1,2,3\}, 1} = & \{ \lambda \in \mathbb{R}^3 \mid (\lambda_1 - \lambda_0, \lambda_2 - \lambda_1, \lambda_3 - \lambda_2) \in V_1 \},\\
        U_{\{1,2,3\}, 2} = & \{ \lambda \in \mathbb{R}^3 \mid (\lambda_1 - \lambda_0, \lambda_2 - \lambda_1, \lambda_3 - \lambda_2) \in V_2 \},\\
                U_{\{1,2,3\}, 3} = & \{ \lambda \in \mathbb{R}^3 \mid (\lambda_1 - \lambda_0, \lambda_2 - \lambda_1, \lambda_3 - \lambda_2) \in V_3 \}.\\
    \end{split}
    \] 
    }
   }

\end{figure}

\begin{definition}[Partitions of unity for divided differences]\label{Dfn=POI}
Let $(\widetilde{\theta}_{k,l})_{1 \leq l \leq k}$ be a partition of unity of smooth functions $\Sphere^{k-1} \rightarrow [0,1]$ subordinate to the cover $(\widetilde{U}_{k,l})_{ 1 \leq l \leq k}$. For  $\xi \in \mathbb{R}^k \backslash \{ 0 \} $ set 
\[
\widetilde{\theta}_{k,l, \mathbb{R}_{>0}}(\xi) :=   \widetilde{\theta}_{k,l}  ( \xi /   \Vert \xi \Vert_2  ).
\]
For  $I$  as in \eqref{Eqn=LabelsofI} and for  $\lambda \in \mathbb{R}^{n+1} \backslash \Delta_{I,n}$, we set 
\[
\begin{split} 
\theta_{I,i_{l}}(\lambda) :=    \widetilde{\theta}_{k,l, \mathbb{R}_{>0}}(Q_I(\lambda)) = \widetilde{\theta}_{k,l, \mathbb{R}_{>0}}(\lambda_{i_1} - \lambda_{i_0}, \ldots, \lambda_{i_k} - \lambda_{i_{k-1}}).
\end{split}
\]  
\end{definition}

\begin{lemma}
The functions  $(\theta_{I,i_l})_{ 1 \leq l \leq k}$ form a partition of unity of functions subordinate to the cover $(U_{I,i_l})_{1 \leq l \leq k}$ of $\mathbb{R}^{n+1} \backslash \Delta_{I,n}$ that are moreover smooth on their domain $\mathbb{R}^{n+1} \backslash \Delta_{I,n}$. 
\end{lemma}
\begin{proof}
The cover  $(\widetilde{\theta}_{k,l, \mathbb{R}_{>0}})_{ 1 \leq l \leq k}$ is a partition of unity subordinate to $( \mathbb{R}_{>0} \cdot \widetilde{U}_{k,l})_{1 \leq l \leq k}$ as it is a homogeneous extension of the partition of unity $(\widetilde{\theta}_{k,l})_{1 \leq l \leq k }$ of $(\widetilde{U}_{k,l})_{ 1 \leq l \leq k}$. Then as $Q_I$ in \eqref{Eqn=QMapNew} is a continuous surjection mapping $U_{I,i_l}$ onto $\mathbb{R}_{>0} \cdot \widetilde{U}_{k,l}$ it follows that  $(\theta_{I,i_l} = \widetilde{\theta}_{k,l, \mathbb{R}_{>0}} \circ Q_I)_{ 1 \leq l \leq k}$ is a partition of unity subordinate to  $(U_{I,i_l})_{1 \leq l \leq k}$.  Finally, $\theta_{I,i_l}$ is smooth as both $Q_I$ and $\widetilde{\theta}_{k,l, \mathbb{R}_{>0}}$ are smooth. 
\end{proof}

\subsection{Auxiliary functions}\label{subsec: sec_5_aux_function}

\begin{definition}\label{Dfn=FSets}
Let $0 \leq k \leq n-1$. We define $\mathcal{F}_{n,k}$ as the set of all sequences
\[
F = (i_1, \sigma_1,  i_2, \sigma_2,  \ldots, i_k,  \sigma_{k})
\]
where $\sigma_1, \ldots, \sigma_k \in \{ +, - \}$ and where $i_1, \ldots, i_k \in \{0, \ldots, n\}$ satisfy the following inductive properties:
\begin{enumerate}
    \item we define $I_{F,1} := \{ 0, \ldots, n\}$ and let  $i_1$ be an   element of  $I_{F,1}$ not equal to 0,   
    \item for $1 \leq l \leq k$ we define inductively 
    \[ 
I_{F,l+1} := \left\{
\begin{array}{ll}
I_{F,l} \backslash  j_s  & \textrm{if } \sigma_{l} = +, \\
I_{F,l} \backslash   j_{s-1} & \textrm{if } \sigma_{l} = -, \\
\end{array}
\right. \qquad
\]
where  $I_{F,l} = \{  j_0 < \ldots < j_{n+1-l}  \}$ and $j_s = i_l$,
\item  let $i_{l+1}$ be an element of $I_{F,l+1}$ not equal to the smallest element of  $I_{F,l+1}$. 
\end{enumerate}
 We will write 
\[
F_l := j_s =i_l, \quad F_{l}^- := j_{s-1}, \quad F_{l}^+ := j_{s+1}.
\] 
In other words, $F_l^-$  is the largest index in $I_{F,l}$ that is smaller than $i_l$ and $F_l^+$ is the smallest index in $I_{F,l}$ that is larger than $i_l$. If $s = n+1-l$, i.e.\ if $F_l$ is the largest element of $I_{F,l}$, we use the cyclic definition  $F_{l}^+ = j_{0}$. In case $k=l = n-1$, the set $I_{F, n}$ consists of only two elements and we write $I_{F, n} = \{ F_{n}^-, F_{n} \}$. The symbol $F_n^+$ shall never occur in our expressions.  
\end{definition}
  
\begin{example}
Consider $n = 4$, $k=3$, and let $F= (1, +, 2,-, 4,-)$. Then we have 
\[
\begin{array}{ll}
I_{F,1} =  \{0, 1, 2, 3, 4\}, & F_1 = 1,\; F_1^- = 0,\; F_1^+ =2, \\
I_{F,2} =  \{0,  2, 3, 4\},   & F_2 = 2,\; F_2^- = 0,\; F_2^+ =3, \\
I_{F,3} =  \{  2, 3, 4\},     & F_3 = 4,\; F_3^- = 3,\; F_3^+ =2, \\
I_{F,4} =  \{   2, 4\},       & F_4 = 4,\; F_4^- = 2.  \\
\end{array}
\]
Intuitively, $F$ describes a procedure in which every $F_l$ is selected from the set $I_{F,l}$ such that it is not its smallest element, after which either $F_l$ (if $\sigma_l = +$) or $F_l^-$ (if $\sigma_l = -$) is deleted from  $I_{F,l}$ to obtain the set $I_{F, l + 1}$.  
\end{example}
 
 \begin{definition}[Coordinate differences and their fractions]
For $0 \leq  a, b   \leq n$ and $a \not = b$ we set 
\[
\xi_{a,b}(\lambda ) := \lambda_a - \lambda_b, \quad   \lambda \in \mathbb{R}^{n+1}. 
\] 
Let $F \in\mathcal{F}_{n,k}$ with $1 \leq k \leq n-1$. For $1 \leq l \leq k+1$ and  $\lambda \in \mathbb{R}^{n+1}$ we set
\begin{equation}\label{Eqn=XiFRhoLNew}
\xi_{F,  l}(\lambda)  := \xi_{ F_l,  F_l^-} = \lambda_{F_l}   - \lambda_{ F_l^- }.
\end{equation} 
 Moreover, for $1 \leq l \leq k$ and  $\lambda \in \mathbb{R}^{n+1}$ with $\lambda_{F_l}   \not =  \lambda_{ F_l^- }$ we set
\begin{equation}\label{Eqn=ZetaFRhoL}
\zeta_{F,l}(\lambda) := 
\left\{
\begin{array}{ll}
\frac{\xi_{F_l^+, F_l^-}}{\xi_{F_l, F_l^-}}  =  \frac{\lambda_{F_l^+} - \lambda_{F_l^-}  }{ \lambda_{F_l} -  \lambda_{F_l^-} }, & \textrm{ if } \sigma_{l}= +, \\
\frac{\xi_{F_l, F_l^+}}{\xi_{F_l, F_l^-}}  =  \frac{\lambda_{F_l} - \lambda_{F_l^+}  }{ \lambda_{F_l} -  \lambda_{F_l^- } },  & \textrm{ if } \sigma_{l} = -. \\
\end{array}
\right.
\end{equation}
Note that for $k = n-1$ we obtain $\zeta_{F, n-1} = \frac{\xi_{F,n}}{\xi_{F, n-1}}$. 
\end{definition}

\begin{definition}[{Divided differences with repeating variables}]\label{def: truncated_int_polyn_momenta}
 Given a set ${I \subseteq \{0, \ldots, n\}}$ and a multi-index $\alpha \in \NNOne^{I}$ we use the notation 
 \[
 \lambda_I^{ (\alpha) } := (\lambda_{i}^{(\alpha_i)} ; i \in I).
 \]
 In particular, let $F \in \mathcal{F}_{n,k}$ with $1 \leq k \leq n-1$ and let $\alpha \in \NNOne^{I_{F, k}}$ be such that $\vert \alpha \vert   = n+1$. Then 
 \[
 f[ \lambda_{ I_{F,k}}^{(\alpha)}   ] := f[\lambda_i^{(\alpha_i)} ; i \in I_{F,k} ]
 \]
 as a function of $\lambda \in \mathbb{R}^{n+1}$. 
\end{definition}

The following lemma shows that for every $F \in \mathcal{F}_{n,n-1}$ and $1 \leq k \leq   n$, the collection of difference functions $\{ \xi_{F,l} \mid k \leq l \leq n \}$ forms a basis of all difference functions on $\mathbb{R}^{I_{F,k}}$ as a real vector space. 

\begin{lemma}\label{Lem=SufficientDifferences}
Let $F \in \mathcal{F}_{n,n-1}$ and take $1 \leq k \leq n$. Let  $a,b \in I_{F,k}$, $a \not = b$. The function $\xi_{a,b}$ lies in the {(real)} linear span  of the functions $\xi_{F, l}$, $k \leq l \leq n$. 
\end{lemma}
\begin{proof}
 First notice that for  any set $J \subseteq \{0,1, \ldots, n\}$ and for any $a,b, c, d \in J$, $a \not = b$, $c \not = d$ we have that $\xi_{a,b}$ lies in the linear span of $\xi_{c, d}$ and $\xi_{r,s}$, $r,s \in J \backslash \{ c\}$. 
Indeed, if $\{ a,b\} = \{ c,d \}$ this statement is obvious. If $a,b \not = c$ then $a, b \in J \backslash \{ c\}$ and the statement is also obvious. Now assume that $a = c$, the case that $b = c$ being similar. The case $b = d$ was covered already and so we assume $b \not = d$.  Then   $\xi_{a,b} = \xi_{c,d} + \xi_{d,b}$ and we are done. 

\vspace{0.3cm}
  
\noindent {\bf Claim:} For $a,b \in I_{F,k}$ with $a\not = b$,  $\xi_{a,b}$  lies in the linear span of the functions 
\[
\xi_{F, l}, \quad l = k, \ldots, n.
\]

\vspace{0.3cm}

We prove this inductively from $k= n$ to $k=1$.  If $k = n$ then $I_{F,n}$ consists of the two points~$F_{n}$ and $F_{n}^-$. As $\xi_{F, n} = \lambda_{F_{n}} - \lambda_{F_{n}^-} = \xi_{F_n, F_n^-}$, the statement follows. 

Now assume the statement is proved for $k+1$; we shall prove it for $k$. Let $J = I_{F,k}$. Let~$\sigma_k$ be the $k$-th $\pm$ sign of $F$. If $\sigma_{k} = +$ then let $c = F_k$ and $d = F_k^-$. If $\sigma_{k} = -$ then let~$c = F_k^-$ and~$d = F_k$. Note that $I_{F,k+1} = J \backslash \{ c\}$.  
By the induction hypothesis $\{\xi_{F, l}\mid l = k+1, \ldots, n\}$ spans same space as $\{ \xi_{r,s} \mid r,s \in I_{F,k+1}, r\not = s\}$. Then the claim follows by the first paragraph. 
\end{proof}
 
  \begin{definition}
Let $F \in \mathcal{F}_{n,k}, 1 \leq k \leq n-1$. For $1 \leq l \leq k$ we use the shorthand notation  
\[
U_{F,l} := U_{I_{F,l}, F_l}, \quad \theta_{F, l} := \theta_{I_{F,l}, F_l}.
\]
\end{definition}

\begin{remark}
    The equation of functions \eqref{Eqn=LambdaZetaToXi} below is {\it a priori} only defined on the domain of~${\lambda \in \mathbb{R}^{n+1}}$ for which $\xi_{F,l}(\lambda) \not = 0$, as otherwise $\zeta_{F,l}$ is not defined. However, by using the partition of unity function $\theta_{F,l}$ we may interpret the value of the left hand side of \eqref{Eqn=LambdaZetaToXi} as 0 for such $\lambda$. Therefore, we  view  \eqref{Eqn=LambdaZetaToXi} as an equality of functions on $\mathbb{R}^{n+1} \backslash \{ 0 \}$. 
\end{remark}

\begin{lemma}\label{Lem=ThetaQIsH}
Let $F \in \mathcal{F}_{n, n-1}$. Let $Q$ be a polynomial in $(n-1)$ variables vanishing on~${\mathbb{R}^{n-2} \times \{ 0 \}}$. 
Then there exists $H \in \mathcal{H}_n$ (see Definition \ref{Dfn=Hn})   such that 
\begin{equation}\label{Eqn=LambdaZetaToXi}
 \left( \prod_{l=1}^{n-1}  \theta_{F, l }(\lambda)  \right)     Q(\zeta_{F,1}, \ldots, \zeta_{F,n-1}) =
H( \xi_{F, 1}, \ldots, \xi_{F, n}). 
\end{equation} 
\end{lemma}
\begin{proof}
        We will first express each of the functions 
        \begin{equation} \label{Eqn=ReExpress}
           Q(\zeta_{F,1}, \ldots, \zeta_{F,n-1}) \quad \textrm{ and } \quad    \theta_{F, k }, \quad  k=1, \ldots, n-1, 
        \end{equation}
        in terms of $\xi_{F, 1}, \ldots, \xi_{F,  n}$. 
For $1 \leq k \leq n$, let  
\[
I_{F, k} = \{ i_0 < \ldots < i_{n+1-k} \}. 
\]
Define $l(k)$ by 
\[
F_k = i_{l(k)}.
\]
By Lemma \ref{Lem=SufficientDifferences}, the linear span of the functions on $\mathbb{R}^{n+1}$ given by   $\xi_{i_{1},i_0}, \ldots, \xi_{i_{n+1-k}, i_{n-k}}$ and by~$\xi_{F, k}, \ldots, \xi_{F,n}$ is the same. Therefore, there exists an invertible matrix
\[
T^{(F,k)} = (T^{(F,k)}_{l,j})_{\substack{1\leq l \leq n+1-k, \\ k \leq j \leq n \qquad }} \in M_{n+1-k}(\mathbb{R}),
\]
such that
\[
\xi_{i_l, i_{l-1}} = \sum_{j = k}^n  T^{(F,k)}_{l, j  } \xi_{F, j}. 
\]
Note that  since $\xi_{i_{l(k)},i_{l(k)-1}}=\xi_{F_k,F_k^-}=\xi_{F,k}$, we have 
\begin{equation}\label{Eqn=OneTopLeft}
T^{(F,k)}_{l(k), k} = 1 \quad \textrm{ and } \quad T^{(F,k)}_{l(k), j} = 0,\; j \not = k.
\end{equation}
Set 
\[
 \widetilde{ \widetilde{\theta}}_{k,l, \mathbb{R}_{>0}}(\xi_1, \ldots, \xi_n)  =    \widetilde{\theta}_{ n+1-k,l, \mathbb{R}_{>0} } \circ T^{(F,k)}(\xi_k, \ldots, \xi_n).
\]
Then by the construction of $ \widetilde{\theta}_{k,l , \mathbb{R}_{>0} }$  in Definition~\ref{Dfn=POI} we have
\begin{equation}\label{Eqn=ThetaXiExpression}
\theta_{F,k}(\lambda)  =  \widetilde{\theta}_{n+1-k,l(k) , \mathbb{R}_{>0}} (\xi_{i_1, i_{0}}, \ldots, \xi_{i_{n+1-k}, i_{n-k}} ) = \widetilde{ \widetilde{\theta}}_{k,l(k) , \mathbb{R}_{>0}} (\xi_{F,1}, \ldots, \xi_{F,n}).
\end{equation}

 Next, recall that  
 \begin{equation}\label{Eqn=ZetaRecall}
 \zeta_{F,  l} = \frac{ \xi_{a,b}  }{\xi_{F,l}},
 \end{equation}
 where $(a,b)$  are described in \eqref{Eqn=ZetaFRhoL} and are in any case elements of $I_{F, l+1}$. Hence, by Lemma \ref{Lem=SufficientDifferences} the numerator of \eqref{Eqn=ZetaRecall} is a linear combination of vectors  $\xi_{F, j}$, $j = l+1, \ldots, n$. Therefore there exists an upper triangular matrix
 \[
R^{(F)} = (R^{(F)}_{l,j})_{1 \leq l \leq n-1, 2 \leq j \leq n} \in M_{n-1}(\mathbb{R})
 \]
 such that for $l = 1, \ldots, n-1$ we have
\[
\zeta_{F, l} = \frac{\sum_{j=l+1}^n R^{(F)}_{l,j} \xi_{F, j}}{\xi_{F, l}}.
\]    
In particular, the fact that $\zeta_{F,  n-1} = \xi_{F, n} \slash \xi_{F, n-1}$ yields  $R^{(F)}_{n-1, n} = 1$.    

Therefore, 
\begin{multline}\label{Eqn=QXiExpression}
Q(\zeta_{F, 1} , \ldots, \zeta_{F, n-1})
=
Q( \frac{\sum_{j=2}^n R^{(F)}_{1,j} \xi_{F, j}}{\xi_{F, 1}} ,  \frac{\sum_{j=3}^n R^{(F)}_{2,j} \xi_{F, j}}{\xi_{F, 2}} , \ldots,  \frac{\sum_{j=n-1}^n R^{(F)}_{n-2,j} \xi_{F, j}}{\xi_{F, n-2}}, \frac{ \xi_{F, n}  }{  \xi_{F, n-1} } ) \\
=: \widetilde{Q}(\xi_{F,1}, \ldots, \xi_{F,n}).
\end{multline}
We have thus expressed each of the functions \eqref{Eqn=ReExpress} in terms of $\xi_{F,1}, \ldots, \xi_{F,n}$. 

Next we prove the statement of the lemma. 
It remains to show that the function $\widetilde{Q} \Theta$ with 
\begin{equation}\label{eqn: big_theta}
  \Theta =  \prod_{k = 1}^{n-1} \widetilde{\widetilde{\theta}}_{k, l(k), \mathbb{R}_{>0}}
\end{equation}
is contained in $\mathcal{H}_n$. 
By construction the support of $\widetilde{\theta}_{ n+1-k,l(k), \mathbb{R}_{>0}}$ with $k = 1, \ldots, n-1$ does not intersect the hyperplane
\[
\{  \xi \in \mathbb{R}^{ n+1-k} \backslash \{ 0 \} \mid  \xi_{l(k)} = 0 \}.
\]
 Hence by  
 \eqref{Eqn=OneTopLeft},  the support of $\widetilde{\widetilde{\theta}}_{k,l(k) , \mathbb{R}_{>0}}$ does not intersect the set
\[
L_ k := \{  \xi \in \mathbb{R}^{n}  \backslash \{ 0 \}   \mid  \xi_{k} = 0 \}.
\]
It follows that the support of the homogeneous function $\Theta$ defined in~\eqref{eqn: big_theta}
intersects none of the sets  $L_k$, $k=1, \ldots, n-1$. 

Next note that $\widetilde{\widetilde{\theta}}_{k,l(k) , \mathbb{R}_{>0}}$ is smooth in every point of  $\mathbb{R}^n$ that is not contained in  $\mathbb{R}^{k-1} \times \{ 0 \}$, where $0 \in \mathbb{R}^{n-k+1}$. Then certainly it is smooth on the support of $\Theta$.

As the support of the homogeneous function $\Theta$ does not intersect the lines $L_k$ for~${k =1, \ldots, n-1}$, it follows that the function $\widetilde{Q}$ defined through \eqref{Eqn=QXiExpression} is bounded and smooth on the support of~$\Theta$. Moreover, $\widetilde{Q}(\xi_1, \ldots, \xi_{n-1}, 0) = 0$ by the assumption that $Q(s_1, \ldots, s_{n-1}, 0 ) = 0$ for all~${s_1, \ldots, s_{n-1} \in \mathbb{R}}$. 

 Altogether we conclude that $\widetilde{Q} \Theta$  is in $\mathcal{H}_n$. 
\end{proof}

\subsection{Reduction formulae}\label{Sect=Reduction} 

Consider the set of variables whose coordinates are all different:
\[
D_{n}  := \left\{ (\lambda_0, \ldots, \lambda_n) \in \mathbb{R}^{n+1} \mid \forall \:  0 \leq i \not = j \leq n: \lambda_i \not = \lambda_j   \right\} \subseteq \mathbb{R}^{n+1}.
\]
By restricting our domain of definition to this set we can assure that the functions $\zeta_{F, l}$ are well-defined for any $F \in \mathcal{F}_{n,k}$, $1 \leq l \leq k$.
\begin{proposition}\label{Prop=KeyDecompositionBrandNew}
For $1 \leq k \leq n-1$ there exist $k$-variable polynomials $Q_{F,k,\alpha}$,  
indexed by~${F \in \mathcal{F}_{n,k}}$ and $\alpha \in \NNOne^{I_{F, k+1}}$, such that 
\begin{equation}\label{Eqn=FactorisatonProperty}
Q_{F, k, \alpha}(\zeta_1, \ldots, \zeta_{l-1}, 0, \zeta_{l+1}, \ldots, \zeta_k) = 0, \quad 1 \leq l \leq k, \:  \zeta_i \in \mathbb{R}, 
\end{equation}
and such that {on $D_n$ we have}
\begin{equation}\label{Eqn=CoreExpansion}
f^{[n]}=  \sum_{ F \in \mathcal{F}_{n,k} }   \sum_{\substack{ \alpha \in \NNOne^{I_{F, k+1}}, \\ \vert \alpha \vert = n +1  }   } \left( \prod_{l=1}^k  \theta_{F, l }(\lambda)  \right)    Q_{F, k , \alpha}(\zeta_{F,1}, \ldots, \zeta_{F,k})   f[\lambda_{I_{F,k+1}}^{(\alpha)}].
\end{equation}
\end{proposition}

\begin{proof}
The proof is essentially a repeated application of the reduction formulae from Section \ref{Sect=ReductionFormula} and proceeds by induction on $k$. All functions that occur in this proof shall be considered as functions on $D_n$. 

The induction starts at $k= 0$, in which case the statement of the proposition trivially reduces to $f^{[n]}(\lambda) = f[\lambda_{0}, \ldots, \lambda_n]$, which holds by definition.

We shall now prove the proposition for $k+1$, assuming that it holds for $0 \leq k \leq n-2$. Consider the expansion \eqref{Eqn=CoreExpansion}. Fix for now an index $F \in \mathcal{F}_{n,k}$ and write $F = (i_1, \sigma_1, \ldots, i_k, \sigma_k)$ 
as in Definition \ref{Dfn=FSets}. Moreover, fix the multi-index $\alpha \in \NNOne^{I_{F, k+1}}$ with $\vert \alpha \vert = n+1$. 
Now for  $i \in I_{F, k+1}$ and $\sigma = \pm$ we introduce the following notation.  We set
\[ 
F(i, \sigma) := (i_1, \sigma_1, \ldots, i_k, \sigma_k, i,\sigma ) \in \mathcal{F}_{n,k +1}.
\]
Then we set 
\[
i := F(i, \sigma)_{k+1}, \quad i^- := F(i, \sigma)^-_{k+1}, \quad i^+ := F(i, \sigma)^+_{k+1}. 
\]
We also define 
\[
\widetilde{I}_{F,k+1} := I_{F,k+1} \backslash \{i^-, i , i^+ \}, \quad  \widetilde{\alpha} := (\alpha_i, i \in \widetilde{I}_{F,k+1}).
\]
Then note that  
\[
\zeta_{F(i,\sigma), k+1}(\lambda) = 
\left\{
\begin{array}{ll}
\frac{\lambda_{i^-} - \lambda_{i^+}  }{ \lambda_{i^-} -  \lambda_{i} }, & \textrm{ if } \sigma= +, \\
\frac{\lambda_{i} - \lambda_{i^+}  }{ \lambda_{i} -  \lambda_{i^- } },  & \textrm{ if } \sigma = -. \\
\end{array}
\right.
\] 
Let $j$ be the smallest element of $I_{F, k+1}$.  Now Corollary~\ref{Cor=AlgebraicDecomposition} gives
\begin{equation}\label{Eqn=VeryComplicatedInductionStep}
\begin{split} 
  f[ \lambda_{I_{F, k+1}}^{ (\alpha )} ] 
 = & \sum_{i \in I_{F,k+1} \backslash \{ j \}}  \theta_{I_{F,k+1}, i}(\lambda)  f[ \lambda_{I_{F, k+1}}^{ (\alpha )} ]   \\
 = &\sum_{i \in I_{F,k+1} \backslash \{ j \} }  \theta_{I_{F,k+1}, i}(\lambda) \left( \sum_{l=0}^{\alpha_{i^-} - 1} p_{\alpha_{i}, l}\left( \frac{\lambda_{i^-} - \lambda_{i^+} }{ \lambda_{i^-} - \lambda_{i}  } \right)  f[\lambda_{i^-}^{(\alpha_{i^-} -l)},   \lambda_{i^+ }^{(\alpha_{ i } +  \alpha_{ i_+} +l)},   \lambda_{ \widetilde{I}_{F,k+1}   }^{( \widetilde{\alpha}     ) } ]  \right. \\
 & \qquad\qquad\qquad\qquad\qquad\left. +  \sum_{l=0}^{\alpha_{i} - 1} p_{\alpha_{i^- }, l}\left(  \frac{\lambda_{i} - \lambda_{i^+} }{ \lambda_{ i} - \lambda_{i^- }  }     \right)  f[ \lambda_{i}^{(\alpha_{i} - l)}, \lambda_{ i^+ }^{(\alpha_{i^- } + \alpha_{i^+ } + l )},   \lambda_{ \widetilde{I}_{F, k+1}  }^{(  \widetilde{\alpha} )}     ]  \right) \\
 = & \sum_{ i \in I_{F, k+1} \backslash \{ j \}  }  \sum_{\sigma \in \{ +,-\} } \theta_{I_{F,k+1}, i}(\lambda) \sum_{l=0}^{ N(\alpha, F, i, \sigma )   } 
 p_{ \alpha, F, i, \sigma, l}( \zeta_{F(i,\sigma), k+1}(\lambda) ) f[ \lambda_{ I_{ F(i,\sigma), k+2 }}^{ {(\beta(\alpha,  F, i,\sigma, l  )}  ) }  ],
\end{split}
\end{equation}
where (using the notation of~\eqref{eqn: decomp_polynomials}) we have
\[
 p_{ \alpha, F, i, \sigma, l}   = \left\{
\begin{array}{ll}
 p_{\alpha_i , l}      &    \textrm{ if }  \sigma = +,  \\
 p_{\alpha_{i^-} , l}      &    \textrm{ if }  \sigma = -,
\end{array}
  \right.  \qquad 
N(\alpha, F, i, \sigma) = \left\{
\begin{array}{ll}
 \alpha_{i^-} -1    &    \textrm{ if }  \sigma = +,  \\
\alpha_{ i } -1     &    \textrm{ if }  \sigma = -,
\end{array}
\right.
\]
and 
\[
 \beta(\alpha,  F, i,\sigma, l  )= ( \beta_j (\alpha,F,i,\sigma, l); \;j \in I_{F(i, \sigma), k+2}) 
\]
is the multi-index of order $n+1$ given by 
\[
 \beta_j (\alpha,F,i,\sigma, l)  = \left\{
\begin{array}{ll}
\alpha_{j},   & \textrm{ if }  j \in \widetilde{I}_{F, k+1},  \\
\alpha_{i^-}     - l,   &    \textrm{ if }  \sigma  = +, j = i^-,   \\
\alpha_{i}  + \alpha_{i^+}  + l,  &   \textrm{ if }  \sigma  = +, j = i^+,   \\
\alpha_{i}   - l,   &    \textrm{ if }  \sigma = -, j  = i,  \\
\alpha_{i^-}  + \alpha_{i^+ }  + l,  &   \textrm{ if }  \sigma  = -, j = i^+. 
\end{array}
  \right.  
\]
Note that $p_{ \alpha, F, i, \sigma, l}$ and $N(\alpha, F, i, \sigma)$ indeed depend on $F$ as $i^-$ depends on both $i$ and  $F$. Moreover, note that for both choices $\sigma = \pm$  we have  $ p_{ \alpha, F, i, \sigma, l} (0) = 0$, as $\alpha_{i}$ and $\alpha_{i^-}$ are nonzero.

Now we sum the expression \eqref{Eqn=VeryComplicatedInductionStep} over all $F \in \mathcal{F}_{n,k}$. The sum over $F$, $i$, and $\sigma$ may be represented by a single summation over $F' = F(i,\sigma)$, i.e.\ summing over all $F'\in \mathcal{F}_{n,k+1}$. Then for 
\[
F':=  F(i, \sigma)  \in \mathcal{F}_{n, k+1} \quad \textrm{ and} \quad \alpha \in \NNOne^{I_{F,k+1}}\textrm{ with }  \vert \alpha \vert = n+1
\]
we set 
\[
N(F', \alpha) := N(\alpha, F, i,\sigma),
\]
and for $1 \leq l \leq N(F', \alpha)$ we set  
\[
\beta(F',\alpha, l) := \beta(\alpha,F, i, \sigma,l). 
\]
Also set $p_{\alpha, F', l} := p_{\alpha, F, i, \sigma, l}$.  
Recall the short hand notation $\theta_{F,k+1} := \theta_{I_{F,k+1}, F_{k+1}}$.
We find {by induction and collecting terms}
\[
\begin{split}
f^{[n]} =  & \sum_{ F' \in \mathcal{F}_{n,k+1} }  \theta_{F', k+1}(\lambda)  
\left( \prod_{l=1}^k  \theta_{F', l }(\lambda)   \right) \\
& \quad \times \quad 
\sum_{\substack{ \alpha \in \NNOne^{I_{F', k+1}}, \\ \vert \alpha \vert = n +1  }   }   \sum_{l=0}^{ N(F', \alpha)   }   Q_{F', k , \alpha}(\zeta_{F',1}, \ldots, \zeta_{F',k})   
 p_{ \alpha, F',   l}( \zeta_{F', k+1}(\lambda) )
 f[ \lambda_{ I_{  F', k+2 }}^{ (\beta( F',\alpha, l  )  ) }  ] \\
 = & \sum_{ F' \in \mathcal{F}_{n,k+1} }    \sum_{\substack{ { \beta } \in \NNOne^{I_{F', k+2}}, \\ \vert { \beta }\vert = n +1  }   } 
\left( \prod_{l=1}^{k+1}  \theta_{F', l }(\lambda)   \right) \\
& \quad \times \quad   \left( 
\sum_{\substack{ \alpha \in \NNOne^{I_{F', k+1}},   \vert \alpha \vert = n +1,\\  l = 0, \ldots,   N(F', \alpha), \\
\textrm{ s.t. }  \beta(F',\alpha, l) =  \beta }   }     Q_{F', k , \alpha}(\zeta_{F',1}, \ldots, \zeta_{F',k})   
 p_{ \alpha, F',  l}( \zeta_{F', k+1}(\lambda) ) \right) 
 f[ \lambda_{ I_{ F', k+2} }^{ (   \beta    ) }  ]. 
\end{split}
\]
The latter expression has the desired form. 
\end{proof}

\begin{proposition}\label{Prop=KeyDecompositionStrong_theta} For every $F \in \mathcal{F}_{n,n-1}$ and  $\alpha = (\alpha_-, \alpha_+) \in \NNOne^2$,  $\vert \alpha \vert = n+1$  there exists a function $H_{F, \alpha} \in \mathcal{H}_n$ such that 
such that for every $f \in C^n(\mathbb{R})$ we have,  
\begin{equation}\label{Eqn=HomogeneousFunctMainDec}
  f^{[n]} = \sum_{ \substack{F \in \mathcal{F}_{n, n-1},\\ \alpha = (\alpha_-, \alpha_+) \in \NNOne^2,  
  \vert \alpha \vert = n+1.  } } 
   H_{F, \alpha}( \xi_{F, 1}, \ldots, \xi_{F, n})  
    f[  \lambda_{F_{n}^-}^{( \alpha_-)}, \lambda_{F_{n}}^{(\alpha_+)}  ]. 
\end{equation}
\end{proposition}

\begin{proof} 
We apply Proposition~\ref{Prop=KeyDecompositionBrandNew} for $k = n-1$. This yields that on $D_n$ we have 
\begin{equation}\label{Eqn=FullStepNew}
\begin{split}
f^{[n]}(\lambda) =  \sum_{ \substack{F \in \mathcal{F}_{n, n-1},\\ \alpha \in \NNOne^{I_{F, n}}, \vert \alpha \vert = n+1} }   \left( \prod_{l=1}^{n-1}  \theta_{F, l }(\lambda)  \right)     Q_{F, n-1 , \alpha}(\zeta_{F,1}, \ldots, \zeta_{F,n-1})   f[\lambda_{I_{F, n}}^{(\alpha)}].
 \end{split}
 \end{equation}
Now note that $I_{F, n}$ contains only two elements and by definition is given by $\{ F_n^-, F_n\}$. Hence, we can represent the summands as $\alpha = (\alpha_-, \alpha_+)$ with $\alpha_- + \alpha_+ = n+1$. 
The statement now follows from Lemma \ref{Lem=ThetaQIsH}.   
     \end{proof}
 
\begin{remark}\label{Rmk=Matching}
We recall that $\xi_{F, n}(\lambda) = \lambda_{F_n} - \lambda_{F_n^-}$. Hence the last variable of the function~$H_{F, \alpha}$ occurring in \eqref{Eqn=HomogeneousFunctMainDec} depends on the same coordinates as the divided difference with repeated variables~$f[  \lambda_{F_{n}^-}^{( \alpha_-)}, \lambda_{F_{n}}^{(\alpha_+)}  ]$. This is one of the conceptual parts of our proof that leads to optimality of the constants.     
\end{remark}

\section{Main result: Upper bound and proof of Theorem A}\label{sect: upper_bound_main_result}

In this section we prove   upper bounds on the norm of Schur multipliers of divided differences,  using the decompositions from Sections \ref{Sect=Homogeneous} and \ref{Sect=DivDiffDecomposition}. In the linear case this was first proved in~\cite{CMPS},  after which alternative proofs were found in~\cite{CPSZ,CJSZ,CGPT}. 

\vspace{0.3cm}

 \noindent {\bf Convention.} Throughout this section fix $n \in \mathbb{N}, n \geq 2$.

\begin{definition}
For $0 \leq i <  j \leq n$ and $1 < p_1, \ldots, p_n < \infty$ we define
\[
p_{(i ; j)} = \left( \sum_{s=i+1}^{j} \frac{1}{p_s}  \right)^{-1}.
\]
Note that if $1 < p := ( \sum_{s=1}^n p_s^{-1} )^{-1} < \infty$ then also $1 < p_{(i ; j)} <\infty$ for any choice of $i$ and $j$. 
 Recall furthermore the notation 
 \[
 p^\sharp = \max(p, p^\ast) = \max(p, \frac{p}{p-1}).
 \]
\end{definition} 

The next lemma is a standard consequence of the elementary reduction equality for Schur multipliers that can be found in \cite[Lemma 3.2.(iv)]{PSS-Inventiones}. We need to apply this reduction inductively and to smoothen the proof we first streamline our notation.    Let $F \in \mathcal{F}_{n, n-1}$.   Recall that for $1 \leq k \leq n$ we have defined the function  
\begin{equation}\label{Eqn=XiLNew}
\xi_{F, k}(\lambda) =   \lambda_{F_k} -  \lambda_{ F_k^-},
\end{equation}
using the notation of Definition~\ref{Dfn=FSets}. Note that
\[
I_{F, n} :=  \{F_{n}^-, F_{n}  \}
\]
consists of two elements. Let  $\alpha =  (\alpha_-, \alpha_+) \in \NNOne^{2}$ with $\alpha_- + \alpha_+  = n+1$ and consider as before
\begin{equation}\label{Eqn=PhiNew}
f^{[n]}( \lambda_{ I_{F, n} }^{( \alpha )}  ) = f[ \lambda_{F_{n}^-}^{( \alpha_-)}, \lambda_{F_{n}}^{(\alpha_+)} ].
\end{equation}
The following convention leads to a more elegant statement of Lemma \ref{Lem=SchrodingerEstimate} below as well as its proof. This convention shall be used only in  Lemma \ref{Lem=SchrodingerEstimate}.

\begin{remark}\label{Rmk=Notation}
Let $F \in \mathcal{F}_{n,n-1}$ and let $1 \leq k \leq n$. In the statement and proof below we shall encounter the function for $t_k, \ldots, t_n \in \mathbb{R} \backslash \{ 0 \}$, 
\[
 \vert \xi_{F,k} \vert^{it_k} \ldots \vert \xi_{F,  n} \vert^{it_n}  f[ \lambda_{F_{n}^-}^{( \alpha_-)}, \lambda_{F_{n}}^{(\alpha_+)} ].
\]
In the definitions so far this is a function on $\mathbb{R}^{n+1}$ that is constant in the variables indexed by the coordinates that are not in $I_{F, k}$. In Lemma \ref{Lem=SchrodingerEstimate} we shall consider this as a function on~${\mathbb{R}^{I_{F, k}} \simeq \mathbb{R}^{n+2-k}}$ and this way it is the symbol of a $n+1-k$-linear Schur multiplier. Note that there are several lines in the proof where we apply this principle to both
\[
\vert \xi_{F,  k} \vert^{it_k} \ldots \vert \xi_{F, n} \vert^{it_n} f[ \lambda_{F_{n}^-}^{( \alpha_-)}, \lambda_{F_{n}}^{(\alpha_+)} ] \quad \textrm{and} \quad 
\vert \xi_{F,  k+1} \vert^{it_{k+1}} \ldots \vert \xi_{F,  n} \vert^{it_n} f[ \lambda_{F_{n}^-}^{( \alpha_-)}, \lambda_{F_{n}}^{(\alpha_+)} ],
\]
which are then symbols on respectively $\mathbb{R}^{I_{F, k}}$  and $\mathbb{R}^{I_{F, k-1}}$, yielding Schur multipliers that are respectively $n+1-k$-linear and $n-k$-linear. 
\end{remark}

\begin{lemma}\label{Lem=SchrodingerEstimate}
Let $f \in C^n(\mathbb{R})$ be such that $\Vert f^{(n)} \Vert_\infty < \infty$.
Let $F \in \mathcal{F}_{n, n-1}$ and let $\alpha \in \mathbb{N}_{\geq 1}^{2}$. Take $1 \leq k \leq n$ and let  
\[
I_{F, k}   = \{ j_0 < \ldots < j_{n+1-k} \}.
\]
We have for any  $1 < p_1, \ldots, p_{n} < \infty$ such that  $1 <  p := (\sum_{s=1}^n p_s^{-1}  )^{-1}  < \infty$, any tuple ${\epsilon = (\epsilon_k, \ldots, \epsilon_n) \in \{0,1\}^{n+1-k}}$, and any  $t_k, \ldots, t_n \in \mathbb{R} \backslash \{ 0 \}$ that 
\begin{equation}\label{Eqn=NestedEstimate}
\begin{split}
& \Vert \left( \prod_{i=k}^n {\rm sign}(\xi_{F,i})^{\epsilon_i} \right) \vert \xi_{F, k} \vert^{it_k} \ldots \vert \xi_{F, n} \vert^{it_n} f[ \lambda_{F_{n}^-}^{( \alpha_-)}, \lambda_{F_{n}}^{(\alpha_+)} ] ] \Vert_{ \frakm_{ p_{(j_{0}; j_{1} )},   \ldots,  p_{(j_{n-k}; j_{n+1-k} ) } }}    \\
 & \qquad 
  \lesssim_n  \qquad   \Vert f^{(n)} \Vert_\infty  \prod_{l=k}^n |t_l | \left( \prod_{l=k}^{n} p^\sharp_{ (F^-_l  ; F_l ) } \right).
\end{split}
\end{equation}  
\end{lemma}
\begin{proof}
We inductively prove the lemma for decreasing $k$ starting from $k=n$ and going to $k=1$. If $k = n$ then the statement of the lemma becomes
\[
\Vert  {\rm sign}(\xi_{F,n})^{\epsilon_n} \vert \xi_{F, n} \vert^{it_n}  f[ \lambda_{F_{n}^-}^{( \alpha_-)}, \lambda_{F_{n}}^{(\alpha_+)} ]      \Vert_{\mathfrak{m}_{p_{( F^-_n    ;  F_n     
  )}       } }    \lesssim  \Vert f^{(n)} \Vert_\infty  \vert t_n \vert  p^\sharp_{ ( F^-_n    ;  F_n     
  )  }.
\]
This estimate holds true by the H\"ormander-Mikhlin-Schur multiplier theorem, see Theorem \ref{Thm=HMS} and Theorem  \ref{Thm=HMSConditions}. 

Now take $1 \leq k \leq n-1$ and assume that the lemma is proved for $k+1$.  
Let 
\[
x_1 \in \cS_{p_{ (j_{0}; j_{1} )} },\; \ldots, \;x_{n+1-k} \in \cS_{ p_{(j_{n-k}; j_{n+1-k} )} }.
\]
We now have the following four decompositions that each follow from \cite[Lemma 3.2 (iii) and (iv)]{PSS-Inventiones}. We will use the   notational convention of Remark \ref{Rmk=Notation}.  Set again the auxiliary function 
\[
\xi(\lambda) = \lambda_1 - \lambda_2, \quad \lambda = (\lambda_1, \lambda_2) \in \mathbb{R}^2.
\]
We obtain the following.
\begin{enumerate}
\item   If $\sigma_k = +$ and $F_k   \not = j_{n+1-k}$ then
\[
\begin{split}
& T_{ \vert \xi_{F, k} \vert^{it_k,\epsilon_k} \ldots \vert \xi_{F,n} \vert^{it_n,\epsilon_n}  f[ \lambda_{F_{n}^-}^{( \alpha_-)}, \lambda_{F_{n}}^{(\alpha_+)} ]   }( x_1, \ldots, x_{n+1-k} ) \\
= &
T_{ \vert \xi_{F, k+1} \vert^{it_{k+1},\epsilon_{k+1}} \ldots \vert \xi_{F, n} \vert^{it_n,\epsilon_n}  f[ \lambda_{F_{n}^-}^{( \alpha_-)}, \lambda_{F_{n}}^{(\alpha_+)} ]   }( x_1, \ldots, x_{F^-_k },  T_{ \vert \xi \vert^{it_k,\epsilon_k} }(x_{ F_k }) x_{ F_k^+  }    , \ldots, x_{n+1-k} ).
\end{split}
\]
\item  If $\sigma_k =  +$ and $F_k     = j_{n+1-k}$ then
\[
\begin{split}
 & T_{ \vert \xi_{F, k} \vert^{it_k,\epsilon_k} \ldots \vert \xi_{F, n} \vert^{it_n,\epsilon_n} f[ \lambda_{F_{n}^-}^{( \alpha_-)}, \lambda_{F_{n}}^{(\alpha_+)} ]    }( x_1, \ldots, x_{n+1-k} )  \\
= &
T_{ \vert \xi_{F,  k+1} \vert^{it_{k+1},\epsilon_{k+1}} \ldots \vert \xi_{F, n} \vert^{it_n,\epsilon_n} f[ \lambda_{F_{n}^-}^{( \alpha_-)}, \lambda_{F_{n}}^{(\alpha_+)} ]   }( x_1, \ldots, x_{F^-_k} )  T_{ \vert \xi \vert^{it_k,\epsilon_k} }(x_{ F_k }).
\end{split}
\]
\item  If  $\sigma_k = -$ and $F_k  \not = j_1$ then  
\[
\begin{split}
& T_{ \vert \xi_{F, k} \vert^{it_k,\epsilon_k} \ldots \vert \xi_{F, n} \vert^{it_n,\epsilon_n} f[ \lambda_{F_{n}^-}^{( \alpha_-)}, \lambda_{F_{n}}^{(\alpha_+)} ]  }( x_1, \ldots, x_{n+1-k} ) \\
= &
T_{ \vert \xi_{F, k+1} \vert^{it_{k+1},\epsilon_{k+1}} \ldots \vert \xi_{F,  n} \vert^{it_n,\epsilon_n} f[ \lambda_{F_{n}^-}^{( \alpha_-)}, \lambda_{F_{n}}^{(\alpha_+)} ]  }( x_1, \ldots, x_{F^-_k}  T_{ \vert \xi \vert^{it_k,\epsilon_k} }(x_{ F_k  }),  x_{ F_k^+ }    , \ldots, x_{n+1-k} ).
\end{split}
\]
\item  If  $\sigma_k = -$ and $F_k   = j_1$ then 
\[
\begin{split}
& T_{ \vert \xi_{F,k} \vert^{it_k,\epsilon_k} \ldots \vert \xi_{F,n} \vert^{it_n,\epsilon_n} f[ \lambda_{F_{n}^-}^{( \alpha_-)}, \lambda_{F_{n}}^{(\alpha_+)} ]   }( x_1, \ldots, x_{n+1-k} ) \\
= &  T_{ \vert \xi \vert^{it_k,\epsilon_k} }(x_{ 1} ) 
T_{ \vert \xi_{F,  k+1} \vert^{it_{k+1},\epsilon_{k+1}} \ldots \vert \xi_{F,  n} \vert^{it_n,\epsilon_n} f[ \lambda_{F_{n}^-}^{( \alpha_-)}, \lambda_{F_{n}}^{(\alpha_+)} ]  }(   x_2 , \ldots, x_{n+1-k} ).
\end{split}
\]
\end{enumerate}
In each of these four cases the lemma follows from induction and the H\"ormander-Mikhlin-Schur theorem, see Theorem \ref{Thm=HMS} and Theorem  \ref{Thm=HMSConditions}. Indeed, this yields the following estimates. 
\begin{enumerate}
\item We obtain
\[
\begin{split}
& \Vert   \left( \prod_{i=k}^n {\rm sign}(\xi_{F,i})^{\epsilon_i} \right)   \vert \xi_{F, k} \vert^{it_k} \ldots \vert \xi_{F,  n} \vert^{it_n} f[ \lambda_{F_{n}^-}^{( \alpha_-)}, \lambda_{F_{n}}^{(\alpha_+)} ]  \Vert_{\mathfrak{m}_{ p_{(j_{0}; j_{1} )},    \ldots,    p_{(j_{n-k}; j_{n+1-k} )} } }    \\
\leq  &
\Vert    \left( \prod_{i=k+1}^n {\rm sign}(\xi_{F,i})^{\epsilon_i} \right)  \vert \xi_{F,  k+1} \vert^{it_{k+1}} \ldots \vert \xi_{F,  n} \vert^{it_n} f[ \lambda_{F_{n}^-}^{( \alpha_-)}, \lambda_{F_{n}}^{(\alpha_+)} ]  \Vert_{ \mathfrak{m}_{ p_{(j_{0}; j_{1} )}, \ldots,   p_{(F_k^-   ; F_k^+)},    \ldots,   p_{ (j_{n-k}; j_{n+1-k} ) } }  }    \\
& \qquad \times \qquad 
   \Vert {\rm sign}(\xi)^{\epsilon_k} \vert \xi \vert^{it_k} \Vert_{\mathfrak{m}_{ p_{(  F_k^-  ; F_k   )} } }  \\ 
 \lesssim  & \Vert f^{(n)} \Vert_\infty |\prod_{l=k+1}^n t_l | \left( \prod_{l=k+1}^{n} p^\sharp_{ ( F_l^-   ; F_l )}  \right)   |  t_k |      p^\sharp_{( F_k^-  ; F_k  )}.
\end{split}
\]
\item We obtain 
\[
\begin{split}
& \Vert   \left( \prod_{i=k}^n {\rm sign}(\xi_{F,i})^{\epsilon_i} \right)    \vert \xi_{F,  k} \vert^{it_k} \ldots \vert \xi_{F, n} \vert^{it_n}  f[ \lambda_{F_{n}^-}^{( \alpha_-)}, \lambda_{F_{n}}^{(\alpha_+)} ]  \Vert_{ \mathfrak{m}_{p_{(j_{0}; j_{1} )},   \ldots,  p_{(j_{n-k}; j_{n+1-k} ) }}}     \\
\leq  &
\Vert   \left( \prod_{i=k+1}^n {\rm sign}(\xi_{F,i})^{\epsilon_i} \right)  \vert \xi_{F, k+1} \vert^{it_{k+1}} \ldots \vert \xi_{F, n} \vert^{it_n} f[ \lambda_{F_{n}^-}^{( \alpha_-)}, \lambda_{F_{n}}^{(\alpha_+)} ]  \Vert_{\mathfrak{m}_{ p_{(j_{0}; j_{1} )},  \ldots,   p_{(j_{n-k-1}; j_{n-k} )} }}    \\
& \qquad \times \qquad  \Vert  {\rm sign}(\xi)^{\epsilon_k}  \vert \xi \vert^{it_k} \Vert_{ \mathfrak{m}_{ p_{(j_{n-k}; j_{n+1-k} )} }}     \\ 
 \lesssim  &  \Vert f^{(n)} \Vert_\infty | \prod_{l=k+1}^n t_l |   \left( \prod_{l=k+1}^{n} p^\sharp_{ (F^-_l   ;  F_l) }  \right)   |   t_k  |     p^\sharp_{(  F_k^-   ; F_k    )}.
\end{split}
\]
\item We obtain, with $F_k^{--}$ the largest index smaller than $F_k^-$ that is contained in $I_{F, k}$,
\[
\begin{split}
& \Vert   \left( \prod_{i=k}^n {\rm sign}(\xi_{F,i})^{\epsilon_i} \right)    \vert \xi_{F,  k} \vert^{it_k} \ldots \vert \xi_{F,  n} \vert^{it_n} f[ \lambda_{F_{n}^-}^{( \alpha_-)}, \lambda_{F_{n}}^{(\alpha_+)} ]  \Vert_{  \mathfrak{m}_{ p_{(j_{0}; j_{1} )}, \ldots, p_{(j_{n-k}; j_{n+1-k} ) }}}     \\
\leq  & \Vert  \left( \prod_{i=k+1}^n {\rm sign}(\xi_{F,i})^{\epsilon_i} \right) \vert \xi_{F,  k+1} \vert^{i t_{k+1}} \ldots \vert \xi_{F,  n} \vert^{it_n} f[ \lambda_{F_{n}^-}^{( \alpha_-)}, \lambda_{F_{n}}^{(\alpha_+)} ]  \Vert_{ \mathfrak{m}_{ p_{(j_{0}; j_{1} )} , \ldots, p_{(F_k^{--} ; F_k)}, \ldots,    p_{(j_{n-k}; j_{n+1-k} ) }}}  \\
& \qquad \times \qquad  \Vert  {\rm sign}(\xi)^{\epsilon_k} \vert \xi \vert^{it_k} \Vert_{   \mathfrak{m}_{p_{(  F_k^- ; F_k   )} }  }  \\ 
 \lesssim  & \Vert f^{(n)} \Vert_\infty \  |  \prod_{l=k+1}^n t_l | \left( \prod_{l=k+1}^{n} p^\sharp_{ ( F^-_l   ;  F_l  )}  \right)    | t_k |     p^\sharp_{ ( F^-_k  ; F_k  )}.
\end{split}
\]
\item We obtain
\[
\begin{split}
& \Vert  \left( \prod_{i=k}^n {\rm sign}(\xi_{F,i})^{\epsilon_i} \right)    \vert \xi_{F, k} \vert^{it_k} \ldots \vert \xi_{F, n} \vert^{it_n}  f[ \lambda_{F_{n}^-}^{( \alpha_-)}, \lambda_{F_{n}}^{(\alpha_+)} ]   \Vert_{  \mathfrak{m}_{ p_{(j_{0}; j_{1} )} , \ldots ,   p_{(j_{n-k}; j_{n+1-k} ) }}}     \\
\leq  &
\Vert \vert  \left( \prod_{i=k+1}^n {\rm sign}(\xi_{F,i})^{\epsilon_i} \right) \xi_{F, k+1} \vert^{it_{k+1}} \ldots \vert \xi_{F, n} \vert^{it_n}  f[ \lambda_{F_{n}^-}^{( \alpha_-)}, \lambda_{F_{n}}^{(\alpha_+)} ]   \Vert_{ \mathfrak{m}_{p_{(j_{1}; j_{2} )} ,  \ldots ,    p_{(j_{n-k}; j_{n+1-k} ) }}}  \\
& \qquad \times \qquad \Vert  {\rm sign}(\xi)^{\epsilon_k} \vert \xi \vert^{it_k} \Vert_{ \mathfrak{m}_{ p_{(j_{0}; j_{1} )} }}   \\ 
 \lesssim  &  \Vert f^{(n)} \Vert_\infty | \prod_{l=k+1}^n t_l | \left( \prod_{l=k+1}^{n} p^\sharp_{ ( F^-_l  ; F_l  )} \right)    | t_k |      p^\sharp_{ ( F^-_k  ; F_k  )}.
\end{split}
\]
\end{enumerate}

In each of the four cases the final estimate is the required estimate of the lemma.

\end{proof}

We will now prove the main result of this paper. 

\begin{theorem}\label{Thm=MainTheorem}
There exists a constant $C_n > 0$ such that for every $f \in C^n(\mathbb{R})$ with $\Vert f^{(n)} \Vert < \infty$ and for every $1 < p, p_1, \ldots, p_n < \infty$ with $p = (\sum_{i=1}^n p_i^{-1})^{-1}$ we have  
\begin{equation}\label{Eqn=MainResult}
\Vert  f^{[n]} \Vert_{  \mathfrak{m}_{p_1, \ldots,  p_n}}   \leq C_n \Vert f^{(n)}\Vert_\infty  \sum_{F \in \mathcal{F}_{n,n-1}}   \prod_{l=1}^{n} p^\sharp_{( F^-_l   ;  F_l   )}. 
\end{equation}
In particular, for {$D_n = C_n \vert \mathcal{F}_{n,n-1} \vert$} we have for any $f \in C^n(\mathbb{R})$ with $\Vert f^{(n)} \Vert < \infty$ and  $1 < p < \infty$ 
\begin{equation}\label{Eqn=SimplifiedMainEstimate}
\Vert  f^{[n]} \Vert_{ \mathfrak{m}_{ np,  \ldots,  np}}   \leq D_n \Vert f^{(n)}\Vert_\infty p^\ast p^n.
\end{equation}
\end{theorem}
\begin{proof}
Let $f \in C^n(\mathbb{R})$. We apply Proposition \ref{Prop=KeyDecompositionStrong_theta} and obtain the decomposition
\[
  f^{[n]} = \sum_{ \substack{F \in \mathcal{F}_{n, n-1},\\ \alpha = (\alpha_-, \alpha_+) \in \NNOne^2, \\ 
  \vert \alpha \vert = n+1.  } } 
   H_{F, \alpha}( \xi_{F, 1}, \ldots, \xi_{F, n})  
    f[  \lambda_{F_{n}^-}^{( \alpha_-)}, \lambda_{F_{n}}^{(\alpha_+)}  ], 
\]
where $H_{F, \alpha} \in \mathcal{H}_n$ (see Definition~\ref{Dfn=Hn}). Therefore it suffices to show that for every $F \in \mathcal{F}_{n, n-1}$ and every 2-index~${\alpha = (\alpha_-, \alpha_+) \in \NNOne^2}$ with  $\vert \alpha \vert = n+1$ and  any $H :=  H_{F, \alpha } \in \mathcal{H}_n$    we have 
\begin{equation}\label{Eqn=CoreOfKeyEstimate}
\Vert   H( \xi_{F, 1}, \ldots, \xi_{F, n} )   f[ \lambda_{F_{n}^-}^{( \alpha_-)}, \lambda_{F_{n}}^{(\alpha_+)} ]  \Vert_{ \mathfrak{m}_{  p_1,    \ldots,   p_n} }   \lesssim_n   \Vert f^{(n)}\Vert_\infty     \prod_{l=1}^{n} p^\sharp_{( F^-_l  ;  F_l ) }. 
\end{equation} 

 \noindent {\it Proof of the key estimate \eqref{Eqn=CoreOfKeyEstimate}}.  
 By Proposition \ref{Prop=FourierTypeDecomposition} there exist Schwartz functions $g_{ \epsilon}$ indexed by~${\epsilon \in \{ 0, 1 \}^n}$ such that for $\xi = (\xi_1, \ldots, \xi_n) \in \mathbb{R}^{n}$, $\xi_1, \ldots, \xi_n \not = 0$  we have
\[
\begin{split}
     H(\xi) 
= & 
\sum_{\epsilon  \in \{0, 1\} } 
\left(   \prod_{i=1}^n {\rm sign}(\xi_i)^{\epsilon_i} \right)  \int_{\mathbb{R}^n}    g_{ \epsilon_1, \ldots, \epsilon_n}(s_1, \ldots,  s_{n-1})  \\
& \qquad \times \qquad  \vert \xi_1 \vert^{ -is_1} \vert \xi_2 \vert^{ i(s_1 - s_2)}  \ldots   \vert \xi_{n-1} \vert^{i(s_{n-2}-s_{n-1})}  \vert \xi_n \vert^{ is_{n-1}}  ds_1 \ldots ds_{n-1}. 
\end{split}
\]
We may therefore apply Lemma \ref{Lem=SchrodingerEstimate} to the function $H( \xi_{F,1}, \ldots, \xi_{F,n} )  f[\lambda_{I_\rho}^{(\alpha)}]$ and find that 
\begin{equation}\label{Eqn=HPhiEstimate}
\begin{split}
        & \Vert H( \xi_{F,1}, \ldots, \xi_{F,n} )    f[ \lambda_{F_{n}^-}^{( \alpha_-)}, \lambda_{F_{n}}^{(\alpha_+)} ]  \Vert_{\mathfrak{m}_{p_1, \ldots, p_n}} \\
\lesssim_n &  \Vert f^{(n)}\Vert_\infty   \left( \prod_{l=1}^{n} p^\sharp_{ ( F^-_l   ; F_l )} \right) \\
& \qquad \times \qquad \sum_{\epsilon \in \{0, 1\}^n } 
 \int_{\mathbb{R}^n}    g_{ \epsilon}(s_1, \ldots,  s_{n-1}) 
 \vert s_1 \vert \vert s_1 - s_2 \vert \ldots \vert s_{n-2} - s_{n-1} \vert \vert  s_{n-1} \vert     ds_1 \ldots  ds_{n-1}. 
\end{split}
\end{equation}
As $g_{ \epsilon}$ is a Schwartz function the integral in \eqref{Eqn=HPhiEstimate} is finite, which proves  \eqref{Eqn=CoreOfKeyEstimate} and concludes the proof of the main estimate.

\vspace{0.3cm}
 
 \noindent {\it Proof of the  final statement of Theorem \ref{Thm=MainTheorem}}.
 For the final claim suppose that    $p_1 = \ldots = p_n = np$.    For $0\leq i < j \leq n$ we then have $p_{(i;j)} = p \frac{n}{j-i}$. In particular,  for any $1 \leq l \leq n$ we have
\[
p_{( F_l^-     ; F_l ) }  \leq np,  
\]
and so the asymptotic upper estimate $O(p^n)$ for $p \rightarrow \infty$ follows from \eqref{Eqn=MainResult}.  Moreover, note that for~$(i,j) \not = (0,n)$ 
we have $p_{(i;j)} = p \frac{n}{j-i} \geq \frac{n}{n-1} > 1$.  Thus if $p \searrow 1$  the terms in  
\eqref{Eqn=MainResult} of the form~$p_{(i;j)}^\sharp$ with   $0\leq i < j \leq n$ and $(i,j) \not = (0,n)$ remain bounded and do not contribute to the asymptotics.  If in a term in  \eqref{Eqn=MainResult} we have that   $F^-_l  = 0$ and  $ F_l   = n$,   then $0$ and $n$ are neighbouring elements in $I_{F,l}$, which can only happen if $I_{F,l} = \{ 0, n \}$. In particular, $I_{F,l}$ has only two elements and this implies that $l = n$. 
Thus for every summand \eqref{Eqn=MainResult}, at most the term indexed by $l = n$ contributes an order $p^{\sharp}$ to the asymptotic behavior of the estimate as $p \searrow 1$. Hence we conclude the  the asymptotic upper estimate $O(p^\ast)$ for $p \searrow 1$. 
\end{proof}

\begin{remark}
The asymptotic estimate \eqref{Eqn=SimplifiedMainEstimate} for $p \searrow 1$ can be used to prove the following extrapolation result. For every $f \in C^{n}(\mathbb{R})$ we have 
\begin{equation}\label{Eqn=Marcinkiewicz}
\Vert M_{f^{[n]}}: \overbrace{ S_n \times \ldots \times S_n}^{n \textrm{ terms}} \rightarrow M_{1, \infty} \Vert < \infty, 
\end{equation}
where $ M_{1, \infty}$ is the Marcinkiewicz space of all compact operators $x \in B(H)$ whose decreasing rearrangement $(\mu_t(x))_{t \in [0, \infty) }$ satisfies 
\[
\Vert x \Vert_{M_{1, \infty}} = \sup_{t \in [0, \infty)} \log(1+t)^{-1} \int_0^t \mu_s(x) ds < \infty.
\]
The proof of this standard fact can be found in the bilinear case $n=2$ in \cite[Theorem 7.4]{CaspersReimann} (or for $n=1$ in \cite[Corollary 5.16]{CMPS}). Its generalisation to the $n$-linear case is straightforward and omitted here.    
\end{remark}

\begin{remark}
The question whether in \eqref{Eqn=Marcinkiewicz} the operator ideal $M_{1, \infty}$ can be replaced by the ideal of compact operators $x \in B(H)$ with $\Vert x \Vert_{1, \infty} := \sup_{t \in [0, \infty)} t \mu_t(x) < \infty$ remains open. This question has a positive answer in in the following cases: (1) in the linear case of $n=1$, see \cite{CPSZ}, (2) the case of the generalized absolute value function, see \eqref{Eqn=AbsValue} below,  as was shown in \cite{CSZ-Israel, CPSZ-JOT}. 
\end{remark}

\section{Lower bounds and proof of Theorem B}\label{Sect=LowerBounds}

The aim of this section is to obtain lower bounds for the norms of multilinear Schur multipliers of higher order divided differences acting on Schatten classes. Key to our proof are higher order divided differences of a generalized absolute value function. This function and variations of it have been used before in~\cite{CLPST, CaspersReimann, CSZ-Israel} to provide examples and counterexamples of (un)bounded multilinear Schur multipliers. The proof method here for $p \searrow 1$ is a multilinear version of the bilinear result from~\cite[Theorem B, Theorem 8.4]{CaspersReimann} and the linear result from~\cite{Davies}. The proof for $p \nearrow \infty$ for $n \geq 3$ requires the introduction one extra concept compared to the bilinear case \cite{CaspersReimann}:  Cotlar's identity.  

\subsection{Cotlar's  identity} We recall the discrete version of Cotlar's identity, as well as the equivalent fact that a certain function space 
(see~\eqref{Eqn=HATorus}) forms an algebra over $\mathbb{C}$.  This idea has been known for a long time, for the discrete case we refer to \cite[p.\ 421]{AIB}.

We view the torus $\mathbb{T}$ as the complex unit circle and let $z$ denote the identity map on $\mathbb{T}$. 
Moreover, we let the space of polynomials on $\mathbb{T}$ and its subspace of polynomials of mean zero be denoted by
\[
\mathcal{A}(\mathbb{T}) := {\rm span} \{ z^k \mid k \in \mathbb{Z} \}, \quad \mathcal{A}(\mathbb{T})^{\circ} := {\rm span} \{ z^k \mid k \in \mathbb{Z}\setminus\{0\} \}.
\]
The discrete Hilbert transform is given by
\[
H: \mathcal{A}(\mathbb{T}) \rightarrow \mathcal{A}(\mathbb{T}),\; z^k \mapsto {   - i \sign(k)}  z^k,
\] 
where by convention $\sign(0) = 0$. 

We write $\Re (a+ib) = a$ and $\Im (a+ib) = b$ for the real {and imaginary} part of a complex number or  function. Similarly, we let $\Re \mathcal{A}(\mathbb{T})$ (resp.\ $\Re \mathcal{A}(\mathbb{T})^\circ$)    denote  the real-valued functions in $\mathcal{A}(\mathbb{T})$ ({resp.\ $\mathcal{A}(\mathbb{T})^\circ$}). We shall repeatedly make use of the fundamental fact (see \cite{AIB}) that $H$ preserves the real-valued functions and  that  
\[
\Vert H: L^p(\mathbb{T}) \rightarrow  L^p(\mathbb{T}) \Vert \approx p^\ast p \approx \Vert H: \Re L^p(\mathbb{T})^\circ \rightarrow  \Re L^p(\mathbb{T})^\circ \Vert. 
\]
Let 
\begin{equation}\label{Eqn=HATorus}
\mathcal{HA}(\mathbb{T}) = \{ F \in  \mathcal{A}(\mathbb{T})^\circ \mid   \Im(F) =   H  \Re(F)  \}.
\end{equation}
 The following lemma is an equivalent formulation of Cotlar's identity\footnote{It can  be verified directly by taking $f$ and $g$ equal to $\cos_k = \frac{1}{2}(z^k + z^{-k})$ or $\sin_k = -\frac{i}{2}(z^k - z^{-k})$ with $k \in \mathbb{N}_{\geq 1}$ and using standard trigonometric fomulae.  } 
 \begin{equation}\label{eqn: cotlar}
     {(Hf)(Hg)-fg=H(fH(g)+H(f)g),\quad f,g\in\Re\mathcal{A}(\mathbb{T})^{\circ},}
 \end{equation}
 a proof of which can be found in \cite[p.\ 421]{AIB}.

\begin{lemma}\label{Lem=Multiply}
 $\mathcal{HA}(\mathbb{T})$ is a (non-unital) algebra over $\mathbb{C}$.   
\end{lemma}
\begin{proof}
{By linearity of $H$ and $H^2 = - {\rm Id}$ on $\mathcal{A}(\mathbb{T})^\circ$, it holds that $\mathcal{HA}(\mathbb{T})$ is a vector space over $\mathbb{C}$.}
It remains to show that~$\mathcal{HA}(\mathbb{T})$ is closed under  products. 
Let $F, G \in   \mathcal{HA}(\mathbb{T})$. Then there exist $f,g \in \Re \mathcal{A}(\mathbb{T})^\circ$ such that  
\[
F =  f + i H(f), \quad G = g + i H(g).  
\]
Then 
\[
FG = fg - H(f) H(g) + i (f H(g) + H(f) g ). 
\]
The property $\Im(FG)  =  H( \Re(FG))$ is then equivalent to Cotlar's identity~\eqref{eqn: cotlar}. That $FG$ has mean zero follows as the Fourier coefficients of $F$ and $G$ can only be nonzero for   positive (nonzero) powers of $z$ and hence the same holds for $FG$. 
\end{proof}

The following proposition shows that $H(f H(f)^{k-1})$ and $H(f)^k$ differ by a term that is at most of order $\mathcal{O}(p^{k-1})$.

\begin{proposition}\label{Prop=DifferenceOpnminusone}
For $k \in \mathbb{N}_{\geq 2}$ there exists a constant $C_k >0$ and $k$-linear maps 
\[
R_k: \mathcal{A}(\mathbb{T}) \times \ldots \times \mathcal{A}(\mathbb{T}) \rightarrow \mathcal{A}(\mathbb{T})
\]
such that for every $f \in \Re \mathcal{A}(\mathbb{T})^\circ$ we have
\[
  H(f)^k    =  { k } H(f H(f)^{k-1})  +  R_k(f, \ldots, f), 
\]
and for every $p \geq 2$ we have 
\[
\Vert R_k(f, \ldots, f) \Vert_p \leq C_k p^{k-1} \Vert f \Vert_{kp}^{k}.
\]
\end{proposition}
\begin{proof}
Let $f \in \Re \mathcal{A}(\mathbb{T})^{ \circ}$ so that $F = f + i Hf \in \mathcal{HA}(\mathbb{T})$.  We have 
\[
\begin{split}
\Re (F^k) = & \sum_{\substack{ l=0 \\ l \textrm{ even} }}^k  (-1)^{\frac{l}{2}} f^{k-l} H(f)^{l}  \binom{k}{l}, \\
\Im( F^k) = & \sum_{\substack{ l=0 \\ l \textrm{ odd} }}^k  (-1)^{\frac{l-1}{2}} f^{k-l} H(f)^{l}  \binom{k}{l}. 
\end{split}
\]
By  Lemma \ref{Lem=Multiply} it follows that $F^k \in  \mathcal{HA}(\mathbb{T})$. Hence $\Im (F^k) = H \Re (F^k)$, and applying $H$ to this equality gives  $\Re (F^k) = - H \Im( F^k )$. We argue that  the latter two equalities yield the statement we need to prove.  

In case $k$ is even,  $\Re( F^k) = - H \Im (F^k)$ gives the following equality, where the terms $l=k$ and~$l=k-1$ are on the left hand side of the equality and all other terms are on the right hand side.
\begin{multline*}
  H(f)^{k} -  { k}
 H (f H(f)^{k-1}) \\
 = 
-  \sum_{\substack{ l=0 \\ l \textrm{ even} }}^{k-2}  (-1)^{\frac{l+k}{2}} f^{k-l} H(f)^{l}    \binom{k}{l} -  \sum_{\substack{ l=0 \\ l \textrm{ odd} }}^{k-2}  (-1)^{\frac{l+k-1}{2}} H(f^{k-l} H(f)^{l} ) \binom{k}{l} =: R_k(f, \ldots, f ). 
\end{multline*}

In each summand in $R_k(f, \ldots, f)$, there are at most $k-1$ applications of $H$. Hence there is some constant $C_k >0$ such that for every  $p \geq 2$ we have
\[
\Vert R_k(f) \Vert_p \leq C_k p^{k-1}  \Vert f \Vert_{kp}^{k}.
\]

In case $k$ is odd, $\Im (F^k) = H \Re (F^k)$  similarly gives
\begin{multline*}
  H(f)^{k} - { k}
 H (f H(f)^{k-1}) \\
 = 
  \sum_{\substack{ l=0 \\ l \textrm{ odd} }}^{k-2}  (-1)^{\frac{l+k}{2}} f^{k-l} H(f)^{l}  \binom{k}{l}  +  \sum_{\substack{ l=0 \\ l \textrm{ even} }}^{k-2}  (-1)^{\frac{l+k-1}{2}} H(f^{k-l} H(f)^{l} ) \binom{k}{l} =: R_k(f, \ldots, f ). 
\end{multline*}

Again, as in each summand of $R_k(f, \ldots, f)$ there are at most $k-1$ applications of $H$, we conclude the proposition.   
\end{proof}

\subsection{Generalized absolute value functions}
For $n \in \NNOne$ we define the generalized absolute value function  
\begin{equation}\label{Eqn=AbsValue}
a_n(s) = \vert s \vert s^{n-1},\quad  s \in \mathbb{R}.
\end{equation}
Note that $a_n \in  C^{n-1}(\mathbb{R}) \cap C^n(\mathbb{R} \backslash \{ 0 \} )$ and moreover
\[
a_n^{(n)}(s) = {\rm sign}(s) n!, \quad s \in \mathbb{R} \backslash \{ 0 \}. 
\]
 Therefore by the integral expresssion \cite[Lemma 5.1]{PSS-Inventiones} for divided differences we see that for~$\sigma \in \{-1, 1 \}$ and $s = (s_0, \ldots, s_n) \in \mathbb{R}^{n+1} \backslash \{ 0 \}$ with $\sigma s_0, \ldots, \sigma s_n \geq 0$ we have 
\begin{equation}\label{Eqn=SignDerivative}
a_n^{[n]}(s_0, \ldots, s_n) = \sigma n!.
\end{equation}
We now prove the following asymptotic estimates for the $n$th order divided difference of $a_n$.  Recall that ${\rm sign}(0) = 0$ by convention. We let $\chi_{\geq 0}$ denote the indicator function of the set of non-negative numbers. 

The $\epsilon_i$'s in the next lemma force the limit that we compute to be somewhat nicer, as it avoids any diagonal terms. This makes the presentation of our proof cleaner. 

\begin{lemma}\label{Lem=QLimit}
Let $q \in (0,1)$ and $n \geq 1$. For $i_0, \ldots, i_n \in \NNZero$ we have
 \begin{equation}\label{Eqn=FirstLimit}
\lim_{k \rightarrow \infty}  \lim_{l \rightarrow \infty} a_n^{[n]}(q^{k i_0}, q^{l i_1},  \ldots, q^{l i_{n-1}} , -q^{k i_n}) = n! \: \sign(i_n - i_0). 
\end{equation}
Fix numbers $0 = \epsilon _n = \epsilon_0  < \epsilon_1 < \ldots < \epsilon_{n-1}  < 1$.  
For  $i_0, \ldots, i_n \in \NNZero$ we have
\begin{equation}\label{Eqn=SecondLimitInThm}
\lim_{k \rightarrow \infty}    a_n^{[n]}(q^{k (i_0 + \epsilon_0)}, -q^{k (i_1 + \epsilon_1)}, -q^{k (i_2 + \epsilon_2)}, -q^{k (i_3 + \epsilon_3)},    \ldots, - q^{k (i_{n} + \epsilon_n)}) =  \alpha(i_0, \ldots, i_n) +  \beta(i_0, \ldots, i_n),
\end{equation}
where  
\[
\begin{split}
\alpha(i_0, \ldots, i_n) = & n! \: {\rm sign}(i_n - i_0) \prod_{l=1}^{n-1}  \chi_{\geq  0}(i_l - i_0), \\
\beta(i_0, \ldots, i_n) = &{-} n! \:   \sum_{k=1}^{{n-1}}  \chi_{< 0}( i_k - i_0 )    \prod_{l=1}^{k-1}    \chi_{\geq 0}(i_l - i_0). 
\end{split}
\]
\end{lemma}
\begin{proof}
We first prove~\eqref{Eqn=FirstLimit}. As $a_n^{[n]}$ is continuous except at $0$, it suffices to show that
\begin{equation}\label{Eqn=LimitInduction}
\lim_{k \rightarrow \infty}  a_n^{[n]}(q^{k i_0}, 0, \ldots, 0,    -q^{k i_n})
= n! \: \sign(i_n - i_0).
\end{equation}
We have by \eqref{eqn: insert_one_xi}  and \eqref{Eqn=SignDerivative} that
\[
\begin{split}
a_n^{[n]}(q^{k i_0}, 0, \ldots, 0,  -q^{k i_n}) = & \frac{q^{k i_0}}{ q^{k i_0} + q^{k i_n}} f^{[n]}(q^{k i_0}, 0, \ldots, 0) + \frac{q^{k i_n}}{q^{k i_0} + q^{k i_n}} f^{[n]}(0, \ldots, 0 , -q^{k i_n}) \\
= &   \frac{q^{k i_0} n! }{q^{k i_0} + q^{k i_n}} - \frac{q^{k i_n} n!}{q^{k i_0} + q^{ k i_n}} = 
\frac{(q^{k i_0} - q^{k i_n}) n! }{q^{k i_0} + q^{k i_n}}.  
\end{split}
\] 
Taking the limit $k \rightarrow \infty$ gives the desired formula \eqref{Eqn=LimitInduction}. 

Next we prove~\eqref{Eqn=SecondLimitInThm}. We have by  \eqref{eqn: insert_one_xi} for $l = 1, \ldots, n-1$ that
\begin{equation}\label{Eqn=SecondLimitInThmPrf}
\begin{split}
  & a_n^{[n]}(q^{k (i_0+ \epsilon_0)}, 0, \ldots, 0, -q^{k (i_l + \epsilon_l) },   \ldots, -q^{k (i_n+ \epsilon_n)  }) \\
  = &
  \frac{q^{k (i_0 + \epsilon_0) }}{ q^{k (i_0 + \epsilon_0)}{  +} q^{k (i_l + \epsilon_l)}}    a_n^{[n]}(q^{k (i_0 + \epsilon_0)}, 0, \ldots, 0, -q^{k (i_{l+1} + \epsilon_{l+1}) },   \ldots, -q^{k (i_{n}  + \epsilon_n ) }) \\
    & +\frac{q^{k (i_l + \epsilon_l)   }}{ q^{k (i_l + \epsilon_l)  } {  +} q^{k (i_0 + \epsilon_0) }}  
    a_n^{[n]}(0, 0, \ldots, 0, -q^{k (i_l + \epsilon_l) },   \ldots, -q^{k (i_n + \epsilon_n)  }).
    \end{split}
\end{equation}
Our choice of $\epsilon_i$'s guarantees that {for $(a,b)\neq(0,n)$}, $i_a + \epsilon_a$ is never equal to  $i_b + \epsilon_b$ and moreover  $i_a + \epsilon_a < i_b + \epsilon_b$ if and only if $i_a \leq i_b$.  
So for $l =1, \ldots, n-1$, 
\begin{equation}\label{Eqn=TakingLimits}
\begin{split}
   & \lim_{k \rightarrow \infty}  a_n^{[n]}(q^{k (i_0 + \epsilon_0)}, 0, \ldots, 0, -q^{k (i_{l} +    \epsilon_l)},   \ldots, -q^{k (i_{n}  + \epsilon_n)} ) \\
=  & \chi_{\geq 0}(i_l - i_0)
\lim_{k \rightarrow \infty}  a_n^{[n]}(q^{k (i_0 + \epsilon_0)}, 0, \ldots, 0, -q^{k ( i_{l+1}  + \epsilon_{l+1})},   \ldots, -q^{k (i_{n} +  \epsilon_n)  }) \\ 
& { -}  \chi_{< 0}(i_l - i_0) n!.  
\end{split}
\end{equation}
Applying \eqref{Eqn=TakingLimits} inductively for $l =1, \ldots, n-1$ and then using \eqref{Eqn=LimitInduction} when $l = n$ gives the result. 
\end{proof}

\subsection{Discrete Schur multipliers}  We will require the notion of discrete Schur multipliers and a discretization theorem. Let $n \in \NNOne$ and let  $F \subseteq \mathbb{Z}$. For $s \in F$ let $P_s$ be the orthogonal projection of $\ell^2(F)$ onto $\mathbb{C} \delta_s$. Then for  $\phi \in  \ell^\infty(F^{\times n+1})$ we shall consider the bounded $n$-linear map 
\[
\begin{split}
T_\phi: S_2(\ell^2(F) ) \times \ldots \times S_2(\ell^2(F) ) \rightarrow &  S_2(\ell^2(F) ), \\
(x_1, \ldots, x_n) \mapsto & \sum_{s_0, \ldots, s_{n} \in F} \phi(s_0, \ldots, s_n)  P_{s_0 }x_1 P_{s_1} x_2 P_{s_2} \ldots P_{s_{n-1}} x_n P_{s_n}.  
\end{split}
\]
In case $F = \mathbb{Z}$ we write $\ell^2(\mathbb{Z}) = \ell^2$. Now let $1 < p, p_1, \ldots, p_n < \infty$. In case $T_\phi$ extends boundedly from $S_2(\ell^2(F) ) \cap S_{p_i}(\ell^2(F) )$ to $S_{p_i}(\ell^2(F) )$ where $i = 1, \ldots, n$ then we shall consider $T_\phi$ as  a multilinear map 
\[
T_\phi: S_{p_1}(\ell^2(F) ) \times \ldots \times S_{p_n}(\ell^2(F) ) \rightarrow   S_{p}(\ell^2(F) ). 
\]
It will always be clear from the context whether $T_\phi$ refers to a discrete or continuous Schur multiplier and on which space it acts.  

As a special case of a discrete linear Schur multiplier, we consider 
 the   upper and lower {\it triangular truncations}
\[
T^{\pm}_F: S_p(\ell^2(F))\rightarrow S_p(\ell^2(F)),\; x \mapsto \sum_{\substack{ s,t \in F \\ \pm s < \pm t}  } P_s x P_t. 
\]
We naturally see $S_p(\ell^2(F))$ as a subspace of $S_p(\ell^2)$ and simply write $T^{\pm} = T^{\pm}_{F}$ 
when~$F \subseteq \mathbb{Z}$ is clear from context.

The proof of the following lemma is a straightforward modification of the bilinear case proved in \cite[Lemma 8.3]{CaspersReimann} and therefore we omit it. Note that the most important part of the proof of this lemma is the application of  \cite[Theorem  2.2]{CKV} which is stated already for general multilinear Schur multipliers. 

\begin{proposition}\label{Prop=DeLeeuw}
Let  $q \in (0,1)$. 
Fix numbers $0 = \epsilon_n = \epsilon_0  < \epsilon_1 < \ldots < \epsilon_{n-1}  < 1$.    For $k,l \in \NNZero$ we define 
\[
\begin{split}
 \overbrace{ \NNZero \times \ldots \times \NNZero }^{n+1} \rightarrow  & \mathbb{R}^{n+1}\\
\phi_{k,l}^1: (i_0, i_1, \ldots, i_{n-1}, i_n) \mapsto &  (q^{k i_0}, q^{l i_1}, \ldots, q^{l i_{n-1}}, -q^{k i_n}), \\
\phi_{k}^2: (i_0, i_1, \ldots, i_{n-1}, i_n) \mapsto &  (q^{k (i_0 + \epsilon_0) }, -q^{k (i_1 + \epsilon_1)},  -q^{k (i_2 + \epsilon_2)}, -q^{k (i_3 + \epsilon_3)},  \ldots,   -q^{k (i_n + \epsilon_n)}). \\
\end{split}
\]
Then for every $k,l \in \NNZero$, $j =1,2$, and $1 < p <  \infty$ we have
\[
  \Vert  T_{a_n^{[n]} \circ \phi_{k,l}^j }: S_{np}(\ell^2) \times \ldots  \times  S_{np}(\ell^2) \rightarrow S_p(\ell^2) \Vert 
\leq \Vert  T_{a_n^{[n]}  }: S_{np}  \times \ldots  \times  S_{np}  \rightarrow S_p  \Vert.
\]
\end{proposition}

Now Theorems \ref{Thm=LowerBoundAt1} and \ref{Thm=LowerBoundAtInfty}  give lower bounds for the asymptotics of multilinear Schur multipliers of divided differences for $p \searrow 1$ and $p \rightarrow \infty$ respectively.

\begin{theorem}\label{Thm=LowerBoundAt1} 
There exists a constant $C > 0$ such that for every $n \in \NNOne$ and every  $1 < p < \infty$ we have
\[
C n! p^\ast \leq \Vert T_{a_n^{[n]}}: S_{np} \times \ldots  \times  S_{np} \rightarrow S_p \Vert.
\]
\end{theorem}
\begin{proof}
Let $q \in (0,1)$ and recall the map $\phi_{k,l} := \phi_{k,l}^1$ from Proposition \ref{Prop=DeLeeuw}. By Lemma \ref{Lem=QLimit} we have  
\[
\lim_{k \rightarrow \infty} \lim_{l \rightarrow \infty} (a_n^{[n]} \circ \phi_{k,l})(i_0, \ldots, i_n) =  
\lim_{k \rightarrow \infty} \lim_{l \rightarrow \infty}  a_n^{[n]}(q^{k i_0}, q^{l i_1},  \ldots, q^{l i_{n-1}} , -q^{k i_n}) = n! \: {\rm sign}(i_n - i_0). 
\]
Take $x_1, \ldots, x_n \in S_{np}(\ell^2(F))$ with $F \subseteq \mathbb{N}$ finite.  We see that 
\begin{equation}\label{Eqn=HToLimKNew}
\begin{split}
n! (T^+  - T^-)(x_1 \ldots x_n) = & \sum_{s_0, \ldots, s_n \in F} n! \: {\rm sign}(s_n - s_0)  P_{s_0} x_1 P_{s_1} \ldots P_{s_{n-1}} x_n P_{s_{n}} \\
=  & \lim_{k \rightarrow \infty} \lim_{l \rightarrow \infty}  \sum_{s_0, \ldots, s_n \in F} (a_n^{[n]} \circ \phi_{k,l})(  s_0, \ldots, s_n ) P_{s_0} x_1 P_{s_1} \ldots P_{s_{n-1}} x_n P_{s_{n}} \\  
= & \lim_{k \rightarrow \infty} \lim_{l \rightarrow \infty}  T_{a_n^{[n]} \circ \phi_{k,l}}(x_1, \ldots, x_n),
\end{split}
\end{equation}
where the convergence holds in norm. Now by Proposition \ref{Prop=DeLeeuw} the bound of  $T_{a_n^{[n]} \circ \phi_{k,l}}$ is uniform in~$k,l$. As $\cup_F S_p(\ell^2(F))$ is dense in $S_p(\ell^2)$,   where the union is over all finite $F\subset\mathbb{Z}$, the equality~\eqref{Eqn=HToLimKNew} holds for any $x_1, \ldots, x_n \in S_{np}(\ell^2)$. Hence for $x_1, \ldots, x_n \in S_{np}(\ell^2)$ with norm one we see that  
\[
\begin{split}
n!  \Vert (T^+  - T^-)(x_1 \ldots x_n) \Vert_p \leq & \limsup_{k \rightarrow \infty } \limsup_{l \rightarrow \infty } \Vert T_{a_n^{[n]} \circ \phi_{k,l}}(x_1, \ldots, x_n) \Vert_p \\
\leq & \limsup_{k \rightarrow \infty}\limsup_{l \rightarrow \infty}
\Vert T_{a_n^{[n]} \circ \phi_{k,l}}: S_{np}(\ell^2) \times \ldots  \times  S_{np}(\ell^2) \rightarrow S_p(\ell^2) \Vert.
\end{split}
\]
Taking the supremum over all $x_i$ and using  Proposition \ref{Prop=DeLeeuw} we find that,  
\begin{equation}\label{Eqn=EstimateHWithA}
\begin{split}
n! \Vert (T^+  - T^-): S_{p}(\ell^2)   \rightarrow S_p(\ell^2) \Vert
\leq & \limsup_{k \rightarrow \infty}
\Vert T_{a_n^{[n]} \circ \phi_{k,l}}: S_{np}(\ell^2) \times \ldots  \times  S_{np}(\ell^2) \rightarrow S_p(\ell^2) \Vert \\
\leq & 
\Vert T_{a_n^{[n]}}: S_{np} \times \ldots  \times  S_{np} \rightarrow S_p \Vert.
\end{split}
\end{equation}
By \cite[Theorem IV.8.2 and IV.7.4]{GohbergKrein} there exists a constant $C>0$ such that for any~${1 < p < \infty}$ we have
\[
C p^\ast  \leq  
 \Vert (T^+  - T^-): S_{p}   \rightarrow S_p \Vert,
\]
which together with \eqref{Eqn=EstimateHWithA} concludes the proof.  
\end{proof}

\begin{theorem}\label{Thm=LowerBoundAtInfty} 
For every $n \in \mathbb{N}_{\geq 1}$ there exists a constant $C_n >0$ such that for every~${1 < p < \infty}$ we have
\[
C_n p^n\leq \Vert T_{a_n^{[n]}}: S_{np} \times \ldots  \times  S_{np} \rightarrow S_p \Vert.
\]
\end{theorem}
\begin{proof} 
As in the proof of Theorem  \ref{Thm=LowerBoundAt1} we find by Proposition \ref{Prop=DeLeeuw} and Lemma \ref{Lem=QLimit} that
\begin{equation}\label{Eqn=LimitLeeuw} 
\begin{split}
\Vert T_{\alpha +  \beta}: S_{np}(\ell^2) \times \ldots  \times  S_{np}(\ell^2) \rightarrow S_p(\ell^2) \Vert \leq &
\limsup_{k \rightarrow \infty} \Vert T_{a_n^{[n]} \circ \phi_k^2}: S_{np}(\ell^2) \times \ldots  \times  S_{np}(\ell^2) \rightarrow S_p(\ell^2) \Vert \\
\leq & 
\Vert T_{a_n^{[n]}}: S_{np} \times \ldots  \times  S_{np} \rightarrow S_p \Vert.
\end{split}
\end{equation}
Now $\alpha$ and $\beta$ are of Toeplitz form, meaning that we may write 
\[
\begin{split}
\alpha(i_0, \ldots, i_n)  = \widetilde{\alpha}(i_1 - i_0, \ldots, i_n - i_{n-1}), \\
\beta(i_0, \ldots, i_n)  =  \widetilde{\beta}(i_1 - i_0, \ldots, i_n - i_{n-1}),
\end{split}
\]
where
\[
\begin{split}
\widetilde{\alpha}(t_1, \ldots, t_n) = & n! \: {\rm sign}(t_1 + \ldots + t_n) \prod_{l=1}^{n-1}  \chi_{\geq  0}(t_1 + \ldots + t_l ), \\
\widetilde{\beta}(  t_1, \ldots, t_n) = & { -}n! \:   \sum_{k=1}^{{n-1}}  \chi_{< 0}( t_1 + \ldots + t_k )    \prod_{l=1}^{k-1}    \chi_{\geq 0}( t_1 + \ldots + t_l ).
\end{split}
\]
Applying multilinear transference from Schur to Fourier multipliers (see \cite[Theorem B]{CKV}) gives 
\begin{equation}\label{Eqn=TransferFourierSchur}
  \Vert M_{\widetilde{\alpha} + \widetilde{\beta}}: L^{np}(\mathbb{T}) \times \ldots  \times  L^{np}(\mathbb{T}) \rightarrow L^{p}(\mathbb{T}) \Vert
\leq 
\Vert T_{\alpha + \beta}: S_{np}(\ell^2) \times \ldots  \times  S_{np}(\ell^2) \rightarrow S_p(\ell^2) \Vert,
\end{equation}
where the Fourier multiplier $M_\phi$ with symbol $\phi$ is defined as the $n$-linear map determined by
\[
M_\phi(z^{k_1}, \ldots, z^{k_n}) = \phi(k_1, \ldots, k_n) z^{k_1+ \ldots + k_n}.
\]

In the rest of the proof we use the symbol $\approx_{ \mathcal{O}(p^{n-1}) }$ for an equality that holds up to a constant~$C_p$ that only depends on $p \geq 2$ and is at most of order $\mathcal{O}(p^{n-1})$.  We notice that all expressions below are norms of expressions that are $n$-linear in $f \in \Re \mathcal{A}(\mathbb{T})^\circ$, and in particular that we can normalise~$f$ and assume $\Vert f \Vert_{np} = 1$.   We also emphasise that we consider only $p \geq 2$ and thus do not take into account the asymptotics as $p \searrow 1$.

We have that $M_{\widetilde{\beta}}$ is a composition of at most $n-1$ Hilbert transforms and triangular truncations and hence its norm is of order at most $\mathcal{O}(p^{n-1})$ for $p \rightarrow \infty$. Therefore 
\begin{equation}\label{Eqn=AlphaBeta}
\Vert M_{\widetilde{\alpha}  }: L^{np}(\mathbb{T}) \times \ldots  \times  L^{np}(\mathbb{T}) \rightarrow L^{p}(\mathbb{T}) \Vert  \approx_{ \mathcal{O}(p^{n-1}) }
\Vert M_{\widetilde{\alpha} + \widetilde{\beta}   }: L^{np}(\mathbb{T}) \times \ldots  \times  L^{np}(\mathbb{T}) \rightarrow L^{p}(\mathbb{T}) \Vert. 
\end{equation}
We let $f \in \Re \mathcal{A}(\mathbb{T})^\circ$. We have  {$H =  -i (T - 2 {\rm Id} )$} with $T =  M_{\chi_{\geq 0}}$  and therefore
\begin{equation}\label{Eqn=OrderOne}
\begin{split}
 \Vert H(f T( f T(f  \ldots T(f T(f)) \ldots  ) ) ) \Vert_p \approx_{\mathcal{O}(p^{n-1})} & \Vert H(f H( f T(f   \ldots T(f T(f)) \ldots ) ) )  \Vert_p \\
\approx_{\mathcal{O}(p^{n-1})}  &  \Vert H(f H( f H(f   \ldots T(f T(f)) \ldots  ) ) ) \Vert_p \\
\approx_{\mathcal{O}(p^{n-1})}  & \ldots \\
  \approx_{\mathcal{O}(p^{n-1})}  & \Vert H(f H( f H(f   \ldots H(f H(f)) \ldots ) ) ) \Vert_p. 
  \end{split}
\end{equation}
Now by Proposition \ref{Prop=DifferenceOpnminusone} we have   $H( f H(f)^{k-1}) \approx_{\mathcal{O}(p^{k-1})} H(f)^{k}$ and applying this repeatedly yields 
\begin{equation}\label{Eqn=OrderTwo}
\begin{split}
\Vert H(f H( f H(f   \ldots H(f H(f)) \ldots ) ) )  \Vert_p   \approx_{\mathcal{O}(p^{n-1})} & \Vert H(f H( f H(f   \ldots H(f)H(f) \ldots  ) ) ) \Vert_p \\
    \approx_{\mathcal{O}(p^{n-1})}  & \ldots \\
    \approx_{\mathcal{O}(p^{n-1})} &  \Vert H(f) \ldots H(f) \Vert_p \\
     = \qquad \quad  & \Vert H(f)^n \Vert_p = \Vert H(f) \Vert_{np}^{n}.
  \end{split}
\end{equation}
Combining \eqref{Eqn=OrderOne} and \eqref{Eqn=OrderTwo} yields 
\[
\Vert  M_{ \widetilde{\alpha}  }(f, \ldots, f)  \Vert_p =  \Vert  H(f T( f T(f ( \ldots T(f) \ldots) ) ) ) \Vert_{p}    
 \approx_{\mathcal{O}(p^{n-1})}   \Vert H(f) \Vert_{np}^{n}.
\]
Taking the supremum over all $f \in \Re \mathcal{A}(\mathbb{T})^\circ$ with $\Vert f \Vert_{np} = 1$ 
gives
\[
\Vert M_{\widetilde{\alpha}  }: L^{np}(\mathbb{T}) \times \ldots  \times  L^{np}(\mathbb{T}) \rightarrow L^{p}(\mathbb{T}) \Vert  \approx_{\mathcal{O}(p^{n-1})}    p^n. 
\]
Together with the estimates \eqref{Eqn=LimitLeeuw}, \eqref{Eqn=TransferFourierSchur}, \eqref{Eqn=AlphaBeta} we conclude the proof.

\end{proof}


\begin{thebibliography}{99999999}
  



\bibitem[BiPu98]{BirmanH}
  M. Birman, A.B. Pushnitski,
   \emph{Spectral shift function, amazing and multifaceted},
   Integral Equations Operator Theory {\bf 30} (1998), no. 2, 191--199. \doi{10.1007/BF01238218}



\bibitem[CaHu25]{CaspersHuisman}
   M. Caspers, E. Huisman, 
   \emph{Higher order perturbation estimates in quasi-Banach Schatten spaces through wavelets},
    International Journal of Mathematics (to appear),  arXiv: 2509.06558. 

\bibitem[CJKM23]{CJKM}
   M. Caspers, B. Janssens, A. Krishnaswamy-Usha,  L. Miaskiwskyi,
  \emph{Local and multilinear noncommutative de Leeuw theorems},
   Math. Ann. {\bf 388} (2024), no. 4, 4251--4305. \doi{10.1007/s00208-023-02611-z}.

    
\bibitem[CJSZ20]{CJSZ}
   M. Caspers, M. Junge, F. Sukochev, D. Zanin,
  \emph{BMO-estimates for non-commutative vector valued Lipschitz functions},
   J. Funct. Anal. {\bf 278} (2020), no. 3, 108317. \doi{10.1016/j.jfa.2019.108317}.


\bibitem[CKV22]{CKV}
  M. Caspers, A. Krishnaswamy-Usha, G. Vos,
   \emph{Multilinear transference of Fourier and Schur multipliers acting on non-commutative $L^p$-spaces},
   Canad. J. Math. {\bf 75} (2023), no. 6, 1986--2006. \doi{10.4153/S0008414X2200058X}.


\bibitem[CMPS14]{CMPS}
   M. Caspers, S. Montgomery-Smith, D. Potapov, F. Sukochev,
   \emph{The best constants for operator Lipschitz functions on Schatten classes},
   J. Funct. Anal. {\bf 267} (2014), no. 10, 3557--3579. \doi{10.1016/j.jfa.2014.08.018}
 



\bibitem[CPSZ15]{CPSZ-JOT}
   M. Caspers, D. Potapov, F. Sukochev, D. Zanin,
   \emph{Weak type estimates for the absolute value mapping},
   J. Operator Theory {\bf 73} (2015), no. 2, 361--384. \doi{10.7900/jot.2013dec20.2021}.

 
\bibitem[CPSZ19]{CPSZ}
   M. Caspers, D. Potapov, F. Sukochev, D. Zanin,
   \emph{Weak type commutator and Lipschitz estimates: resolution of the Nazarov-Peller conjecture},
    Am. J. Math. {\bf 141} (2019), no. 3, 593--610. \doi{10.1353/ajm.2019.0019}.





\bibitem[CaRe25]{CaspersReimann}
   M. Caspers, J. Reimann, 
   \emph{On the best constants of Schur multipliers of second order divided difference functions},
   Math. Ann. {\bf 392} (2025), no. 1, 1119--1166. \doi{10.1007/s00208-025-03111-y}




\bibitem[CSZ21]{CSZ-Israel}
   M. Caspers, F. Sukochev, D. Zanin,
   \emph{Weak $(1,1)$ estimates for multiple operator integrals and generalized absolute value functions},
   Israel J. Math. {\bf 244} (2021), no. 1, 245--271. \doi{10.1007/s11856-021-2179-0}.
 

\bibitem[CLM25]{marcinkiewicz_schur}
    C. Chuah, Z. Liu, T. Mei, 
    \emph{A Marcinkiewicz multiplier theory for Schur multipliers}, 
    Analysis \& PDE, to appear.

\bibitem[CLS21]{CLS-AIF}
     C. Coine, C. Le Merdy, F. Sukochev,
   \emph{When do triple operator integrals take value in the trace class?},
    Ann. Inst. Fourier {\bf 71}  (2021), no. 4, 1393--1448. \doi{10.5802/aif.3422}


\bibitem[CLPST16]{CLPST}
  C. Coine, C.  Le Merdy, D. Potapov, F. Sukochev, A. Tomskova,
   \emph{Resolution of Peller’s problem concerning Koplienko-Neidhardt trace formulae},
   Proc. Lond. Math. Soc. {\bf 113} (2016), no. 2, 113--139 . \doi{10.1112/plms/pdw024}

 



\bibitem[CGPT23]{CGPT}
   J.M. Conde-Alonso, A.M. Gonz\'alez-P\'erez, J. Parcet, E. Tablate,
   \emph{Schur multipliers in Schatten-von Neumann classes},
   Ann. of Math. (2) {\bf 198} (2023), no. 3, 1229--1260. \doi{10.4007/annals.2023.198.3.5}

\bibitem[Dav88]{Davies}
   E. Davies,
   \emph{Lipschitz continuity of functions of operators in the Schatten classes},
   J. Lond. Math. Soc. {\bf 37}  (1988), 148--157. \doi{10.1112/jlms/s2-37.121.148}

 
\bibitem[DySk09]{dykema_skripka_jfa}
    K. Dykema, A. Skripka,
    \emph{Higher order spectral shift},  
    J. Funct. Anal. {\bf 257} (2009), no. 4, 1092--1132. \doi{10.1016/j.jfa.2009.02.019}
 
\bibitem[Far67]{Farforovskaya1}
   Y. Farforovskaya,
   \emph{An estimate of the nearness of the spectral decompositions of self-adjoint operators in the Kantorovich-Rubinstein metric},
   Vestnik Leningrad. Univ. {\bf 22} (1967), no. 19, 155--156.


\bibitem[Far68]{Farforovskaya2}
    Y.  Farforovskaya,
    {\emph The connection of the Kantorovich-Rubinstein metric for spectral resolutions of selfadjoint operators with functions of operators}, 
    Vestnik Leningrad. Univ. {\bf 23} (1968), no. 19, 94--97.

\bibitem[Far72]{Farforovskaya3}
    Y.   Farforovskaya,
    \emph{An example of a Lipschitz function of self-adjoint operators with nonnuclear difference under a nuclear perturbation}, 
    Zap. Nauchn. Sem. Leningrad. Otdel. Mat. Inst. Steklov. (LOMI) {\bf 30} (1972), 146--153. (Russian). English transl. in J. Math. Sci. {\bf 4} (1975), 426-–433. \doi{10.1007/BF01084922}



\bibitem[GMN99]{Gesztesy}
  F. Gesztesy, K. Makarov, S. Naboko.
  The spectral shift operator. Mathematical results in quantum mechanics (Prague, 1998),
   59–90, Oper. Theory Adv. Appl., 108, Birkhäuser, Basel, 1999. \doi{10.1007/978-3-0348-8745-8_5}


\bibitem[GoLa25]{GoldsteinLabuschagne}
  S. Goldstein, L. Labuschagne, 
  \emph{Noncommutative measures and $L^p$ and Orlicz Spaces, with Applications to Quantum Physics},
   Oxford University Press 2025. 

\bibitem[GoKr69]{GohbergKrein}
    I.C.  Gohberg, M.G. Krein,
   \emph{Introduction to the theory of linear nonselfadjoint operators},
      Translations of Mathematical Monographs. 18. Providence, RI: American Mathematical Society (AMS). xv, 378 p. (1969).

\bibitem[GPGR]{RieszSchur}
  Adrian Gonz\'alez-P\'erez, Javier Parcet, Jorge P\'erez García, \'Eric Ricard,
   \emph{Riesz-Schur transforms},
    arXiv: 2411.09324.


\bibitem[HNVW16]{AIB}
   T. Hyt\"onen, J. van Neerven, M. Veraar, L. Weis, 
   \emph{Analysis in Banach spaces. Vol. I. Martingales and Littlewood-Paley theory}, 
     Springer, Cham, 2016, xvi+614 pp.
 



\bibitem[Kop84]{kopl_trace}
L. S. Koplienko, 
\emph{Trace formula for perturbations of nonnuclear type}, 
Sibirsk. Mat. Zh. {\bf 25} (1984), 62-71 (Russian). English transl. in Siberian Math. J. {\bf 25} (1984), 735–74. \doi{10.1007/BF00968686}

 


\bibitem[Kre53]{Krein1}
   M.G. Krein,
   \emph{On the trace formula in perturbation theory},
   {\it Mat. Sbornik N.S.} 33(75): 597–626, 1953.


\bibitem[Kre62]{Krein2}
   M.G. Krein.
   \emph{On perturbation determinants and a trace formula for unitary and self-adjoint operators}.
   \textit{Dokl. Akad. Nauk SSSR}. 144: 268--271, 1962.

   \bibitem[Kre64]{KreinConjecture}
 M. G. Krein, 
 \emph{Some new studies in the theory of perturbations of self-adjoint operators},
   First Math. Summer School, Part I (Russian), Izdat.``Naukova Dumk'', Kiev, 1964, pp. 103--187.
 

\bibitem[Lif52]{Lifschitz}
    I.M. Lifschitz.
   \emph{On a problem of perturbation theory}.  \textit{Uspekhi Mat. Nauk.}. 7:171--180, 1952.

\bibitem[McDSu22]{McDonaldSukochev}
   E. McDonald, F. Sukochev, 
   \emph{Lipschitz estimates in quasi-Banach Schatten ideals},
   Math. Ann. {\bf 383} (2022), no. 1--2, 571--619. \doi{10.1007/s00208-021-02247-x}.
 


\bibitem[Pel85]{Peller}
  V.V. Peller,
  \emph{Hankel operators in the theory of perturbations of unitary and selfadjoint operators},
   Funktsional. Anal. i Prilozhen. {\bf 19} (1985), no. 2, 37--51, 96. \doi{10.1007/BF01078390}.

\bibitem[NaPe09]{NazarovPeller}
  F. Nazarov, V. Peller,
  \emph{Lipschitz functions of perturbed operators}, 
   C. R. Math. Acad. Sci. Paris {\bf 347} (2009), no. 15-16, 857--862. \doi{10.1016/j.crma.2009.05.003}.


\bibitem[Par26]{ParcetICM}
  J. Parcet, 
  \emph{The impact of Schur multipliers in harmonic analysis and operator algebras}, 
  Proceedings ICM 2026,  arXiv:2510.17732.  




\bibitem[PSS13]{PSS-Inventiones}
   D. Potapov, A. Skripka, F. Sukochev, 
   \emph{Spectral shift function of higher order},  
   Invent. Math. {\bf 193} (2013), no. 3, 501--538. \doi{10.1007/s00222-012-0431-2}.

\bibitem[PSST17]{PSST17}
   D. Potapov, A. Skripka, F. Sukochev, A. Tomskova,
    \emph{Multilinear Schur multipliers and Schatten properties of operator Taylor remainders},
    Adv. Math. {\bf 320} (2017), 1063--1098. \doi{10.1016/j.aim.2017.09.012}



\bibitem[PoSu11]{PoSu11}
D. Potapov and F. Sukochev, 
\emph{Operator-Lipschitz functions in Schatten–von Neumann classes},
Acta Mathematica {\bf 207} (2011), no. 2, pp. 375--389. \doi{10.1007/s11511-012-0072-8}.

 
\bibitem[Skr10]{skripka_indiana}
    A. Skripka,
    \emph{ Higher order spectral shift, II. Unbounded case.},
    Indiana Univ. Math. J. {\bf 59} (2010), no. 2, 691--706. 

\bibitem[Skr11]{skripka_illinois}
    A. Skripka,
    \emph{Multiple operator integrals and spectral shift},
     Illinois J. Math. {\bf 55} (2011), no. 1, 305--324. \doi{10.1215/ijm/1355927038} 

 \bibitem[SkTo19]{SkripkaTomskova}
 A. Skripka, A. Tomskova, 
  \emph{Multilinear operator integrals}, 
  Lecture Notes in Math., 2250, Springer, Cham, 2019, xi+190 pp. \doi{10.1007/978-3-030-32406-3}


\bibitem[Skr14]{SkripkaSurvey}
 A. Skripka,
 \emph{Taylor approximations of operator functions, in Operator Theory in Harmonic
and Non-commutative Analysis.}
  Operator Theory Advances and Applications, vol. 240 (Birkhäuser, Basel, 2014), pp. 243--256.




\end{thebibliography}
\end{document}